\documentclass[12pt]{article}
\usepackage{hyperref}
\usepackage{amsmath,amsthm,amsfonts,amssymb,amscd}
\usepackage{enumitem}
\usepackage[export]{adjustbox}
\usepackage{xcolor}

\usepackage[margin=3cm]{geometry}
\theoremstyle{definition}
\newtheorem*{defn*}{Definition}
\newtheorem{defn}{Definition}[section]
\newtheorem{remark}[defn]{Remark}

\newtheorem{conj}{Conjecture}

\theoremstyle{theorem}
\newtheorem{prop}[defn]{Proposition}

\newtheorem{thm}[defn]{Theorem}
\newtheorem{lem}[defn]{Lemma}

\def\R{\mathcal{R}}
\def\S{\mathcal{S}}
\def\Q{\mathbb{Q}}
\def\C{\mathbb{C}}
\def\Z{\mathbb{Z}}

\def\eps{\varepsilon}
\def\ra{\rightarrow}

\begin{document}
\extrafloats{100}

\title{Generating sets for the kauffman skein module of a family of Seifert manifolds}

\author{Jos\'e Rom\'an Aranda} 
\author{Nathaniel Ferguson}
%

\begin{center}\Large{Generating sets for the kauffman skein module of a family \\ of Seifert fibered spaces}\\
\large{Jos\'e Rom\'an Aranda and Nathaniel  Ferguson}
\end{center}

\begin{abstract} 
We study spanning sets for the Kauffman bracket skein module $\S(M,\Q(A))$ of orientable Seifert fibered spaces with orientable base and non-empty boundary. As a consequence, we show that the KBSM of such manifolds is a finitely generated $\S(\partial M, \Q(A))$-module. 
\end{abstract}

Skein modules are a useful tool to study 3-manifolds. Roughly speaking, a skein module captures the space of links in a given 3-manifold, modulo certain local (skein) relations between the links.  
The choice of skein relations must strike a careful balance between providing interesting structure and ensuring that the structure is managable \cite{fundamentals_KBSM}. The most studied skein module is the Kauffman bracket skein module, so named because the skein relations are the same relations used in the construction of the Kauffman bracket polynomial.

Let $\mathcal{R}$ be a ring containing an invertible element $A$. The \textbf{Kauffman bracket skein module} of a 3-manifold $M$ is defined to be the $\R$-module $\S(M,\R)$ spanned by all framed links in $M$, modulo isotopy and the skein relations
\begin{center}
(K1): $\includegraphics[valign=c,scale=.45]{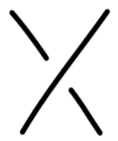}= A\includegraphics[valign=c,scale=.45]{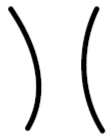}+ A^{-1} \includegraphics[valign=c,scale=.5]{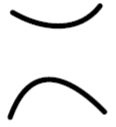}$ 
$ \quad \quad \quad$
(K2): $L\cup \includegraphics[valign=c,scale=.35]{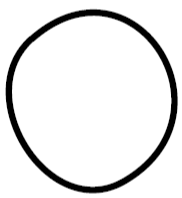} = \left( -A^2-A^{-2}\right) L$.
\end{center}

Throughout this note, when $\mathcal{R}$ is unspecified, it is assumed that $\S(M) = \S(M, \Q(A))$. Since its introduction by Przytycki \cite{KBSM_Przytycki} and Turaev \cite{KBSM_Turaev}, $\S(M,\R)$ has been studied and computed for various 3-manifolds. It is difficult to describe $\S(M,\R)$ for a given 3-manifold, although some results have been found\footnote{As summarized in \cite{fundamentals_KBSM} this is not a complete list of 3-manifolds for which $\S(M,\R)$ is known.}. 
\begin{itemize}
\item $\S(S^3, Z[A^{\pm 1}]) = Z[A^{\pm 1}]$.
\item $\S(S^1\times S^2,\Z[A^{\pm1}])$ is isomorphic to $\Z[A^{\pm1}] \oplus \left( \bigoplus_{i=1}^{\infty} \Z[A^{\pm1}]/(1-A^{2i+4})\right)$ \cite{KBSM_S1S2_PH}.
\item $\S(L(p, q), Z[A^{\pm 1}])$ is a free $Z[A^{\pm 1}]$ module with $\lfloor p/2 \rfloor + 1$ generators \cite{KBSM_lens_PH,non_comm_torus}.
\item $\S(\Sigma \times [0,1], Z[A^{\pm 1}])$ is a free module generated by multicurves in $\Sigma$  \cite{fundamentals_KBSM,confluence}.
\item $\S(\Sigma\times S^1, \Q(A))$ is a vector space of dimension $2^{2g+1}+2g-1$ if $\partial \Sigma=\emptyset$, \cite{KBSM_S1_uppper,KBSM_S1}.
\end{itemize}
In 2019, Gunningham, Jordan and Safronov proved that, for closed 3-manifolds, $\mathcal{S}(M,\C(A))$ is finite dimensional \cite{KBSM_finiteness}.  However, for 3-manifolds with boundary, this problem is still open. 
In \cite{KBSM_infinite}, Detcherry asked versions of a finiteness conjecture for the skein module of knot complements and general 3-manifolds (see Section 3 of \cite{KBSM_infinite} for a detailed exposition).

\begin{conj}[Finiteness conjecture for manifolds with boundary \cite{KBSM_infinite}]\label{conj_finiteness}
Let $M$ be a compact oriented 3-manifold. Then $\mathcal{S}(M)$ is a finitely generated $\mathcal{S}(\partial M, \Q(A))$-module. 
\end{conj}

This paper studies the finiteness conjecture for a large family of Seifert fibered spaces, SFS. Let $\Sigma$ be an orientable surface of genus $g$ with $N$ boundary components. Let $n$, $b$ be non-negative integers with $N=n+b$. For each $i=1, \dots, n$, pick pairs of relatively prime integers $(q_i, p_i)$ satisfying $0<q_i < |p_i|$. The 3-manifold $\Sigma\times S^1$ has torus boundary components with horizontal meridians $\mu_i\subset \Sigma\times \{pt\}$ and vertical longitudes $\lambda_i=\{pt\} \times S^1$. Denote by $M\left(g; b,\{(q_i,p_i)\}_{i=1}^n\right)$ the result of Dehn filling the first $n$ tori of $\partial \left(\Sigma\times S^1\right)$ with slopes $q_i\mu_i + p_i\lambda_i$. Every SFS with orientable base orbifold is of the form $M\left(g; b,\{(q_i,p_i)\}_{i=1}^n\right)$ \cite{hatcher_3M}. The main result of this paper is to establish Conjecture \ref{conj_finiteness} for such SFS. 

\newtheorem*{thm:conj_3_boundary}{Theorem \ref{thm_conj_3_boundary}}
\begin{thm:conj_3_boundary}
Let $\Sigma$ be an orientable surface with non-empty boundary. Then $\mathcal{S}(\Sigma\times S^1)$ is a finitely generated $\mathcal{S}(\partial \Sigma\times S^1, \Q(A))$-module of rank at most $2^{2g+1}-1$. 
\end{thm:conj_3_boundary}

\newtheorem*{thm:finite_SFS}{Theorem \ref{thm_finite_SFS}}
\begin{thm:finite_SFS}
Let $M=M\left(g; b,\{(q_i,p_i)\}_{i=1}^n\right)$ be an orientable Seifert fibered space with non-empty boundary. Suppose $M$ has orientable orbifold base. Then, $\mathcal{S}(M)$ is a finitely generated $\mathcal{S}(\partial M, \Q(A))$-module of rank at most $(2^{2g+1}-1) \prod_{i=1}^n (2q_i-1)$.
\end{thm:finite_SFS}

The following is a more general formulation of the finiteness conjecture. 

\begin{conj}[Strong finiteness conjecture for manifolds with boundary \cite{KBSM_infinite}]\label{conjecture_strong_finiteness}
Let $M$ be a compact oriented 3-manifold. Then there exists a finite collection $\Sigma_i, \dots, \Sigma_k$ of essential subsurfaces $\Sigma_i \subset \partial M$ such that: 
\begin{itemize} 
\item for each $i$, the dimension of $H_1(\Sigma_i, \Q)$ is half of $H_1(\partial M, \Q)$; 
\item the skein module $\mathcal{S}(M)$ is a sum of finitely many subspaces $F_1, \dots, F_k$, where $F_i$ is a finitely generated $\mathcal{S}(\Sigma_i, \Q(A))$-module. 
\end{itemize}
\end{conj}
We are able to show this conjecture for a subclass of SFS. 
\newtheorem*{thm:conj2_SFS}{Theorem \ref{conj2_SFS}}
\begin{thm:conj2_SFS}
Seifert fibered spaces of the form $M\left(g; 1,\{(1,p_i)\}_{i=1}^{n}\right)$ satisfy Conjecture \ref{conjecture_strong_finiteness}. In particular, Conjecture \ref{conjecture_strong_finiteness} holds for $\Sigma_{g,1}\times S^1$. 
\end{thm:conj2_SFS}
The techniques in this work are based on the ideas of Detcherry and Wolff in \cite{KBSM_S1}. For simplicity, we set $\mathcal{R} = \Q(A)$ by default, even though our statements work for any ring $\mathcal{R}$ such that $1 - A^{2m}$ is invertible for all $m > 0$. It would be interesting to see if the generating sets in this work can be upgraded to verify Conjecture \ref{conjecture_strong_finiteness} for all Seifert fibered spaces. Although there is no reason to expect the generating sets to be minimal, we wonder if the work of Gilmer and Masmbaum in \cite{KBSM_S1_uppper} could yield similar lower bounds.

\textbf{Outline of the work.} 
The sections in this paper build-up to the proof of Theorems \ref{thm_finite_SFS} and \ref{conj2_SFS} in Section \ref{section_SFS}. Section \ref{section_prelims} introduces the arrowed diagrams which describe links in $\Sigma\times S^1$. We show basic relations among arrowed diagrams in Section \ref{section_relations_pants}. Section \ref{section_planar_case} proves that $\S(\Sigma_{0,N}\times S^1)$ is generated by boundary parallel diagrams. 
Section \ref{section_non_planar_case} studies the positive genus case $\S(\Sigma_{g,N}\times S^1)$; we find a generating set over $\Q(A)$ in Proposition \ref{thm_boundary_case}. 
In Section \ref{section_SFS}, we describe global and local relations between links in the skein module of Seifert fibered spaces. We use this to build generating sets in Section \ref{section_proof_41}.

\textbf{Acknowledgments.} This work is the result of a course at and funding from Colby College. The authors are grateful to Puttipong Pongtanapaisan for helpful conversations and Scott Taylor for all his valuable advice.


\section{Preliminaries}\label{section_prelims}

Most of the arguments in this paper will focus on finding relations among links in $\Sigma\times S^1$ for some compact orientable surface $\Sigma$. The main technique is the use of arrow diagrams introduced by Dabkowski and Mroczkowski in \cite{KBSM_pants_S1}. \\
An \textbf{arrow diagram} in $\Sigma$ is a generically immersed 1-manifold in $\Sigma$ with finitely many double points, together with crossing data on the double points, and finitely many arrows in the embedded arcs. Such diagrams describe links in $\Sigma\times S^1$ as follows: Write $S^1 = [0,1]/\left(0\sim 1\right)$. Lift the knot diagram in $\Sigma\times \{1/2\}$ away from the arrows to a union of knotted arcs in $\Sigma\times [1/4,3/4]$, and interpret the arrows as vertical arcs intersecting $\Sigma\times\{1\}$ in the positive direction. We can use the surface framing on arrowed diagrams to describe framed links in $\Sigma\times S^1$. 

\begin{figure}[h]
\centering
\includegraphics[scale=.5]{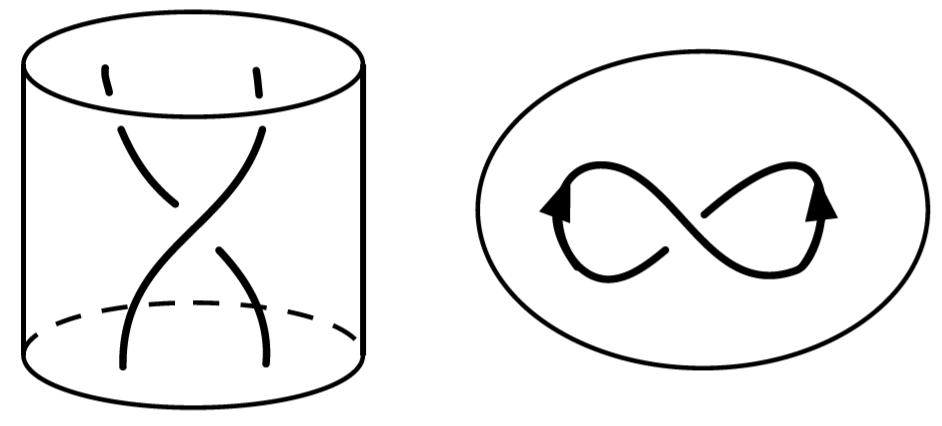}
\caption{Example of arrowed diagram.}
\label{fig_arrowed_diagram}
\end{figure}

Arrowed diagrams have been used to study the skein module of $\Sigma_{0,3}\times S^1$ \cite{KBSM_pants_S1}, prism manifolds \cite{KBSM_prism}, the connected sum of two projective spaces \cite{KBSM_RP3RP3}, and $\Sigma_g \times S^1$ \cite{KBSM_S1}.

\begin{prop}[\cite{KBSM_pants_S1}]
Two arrowed diagrams of framed links in $\Sigma \times S^1$ correspond to isotopic links if and only if they are related by standard Reidemeister moves $R'_1$, $R_2$, $R_3$ and the moves
\begin{center}
(R4): $\includegraphics[valign=c,scale=.5]{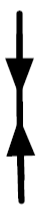}\sim \includegraphics[valign=c,scale=.5]{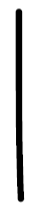}\sim \includegraphics[valign=c,scale=.5]{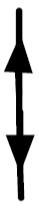}  \quad \quad \quad \quad $
\text{(R5):} $\includegraphics[valign=c,scale=.5]{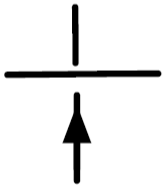} \sim \includegraphics[valign=c,scale=.5]{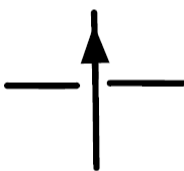}$.
\end{center}
\end{prop}

From relation $R_4$, we only need to focus on the total number of the arrows between crossings. We will keep track of them by writting a number $n\in \Z$ next to an arrow. Negative values of $n$ correspond to $|n|$ arrows in the opposite direction.

Throughout this work, a \textbf{simple arrowed diagram} (or arrowed multicurve) will denote an arrowed diagram with no crossings. 
A simple closed curve in $\Sigma$ will be said to be \textbf{trivial} if it bounds a disk. We will sometimes refer to trivial curves bounding disks disjoint from a given diagram as unknots. Loops parallel to the boundary will not be considered trivial. A simple closed curve will be \textbf{essential} if it does not bound a disk nor is parallel to the boundary in $\Sigma$.

We can always resolve the crossings of an arrowed diagram via skein relations. Thus, every element in $\mathcal{S}(\Sigma\times S^1)$ can be written a $\Z[A^{\pm1}]$-linear combination of arrowed diagrams with no crossings. The following equation will permit us to disregard arrowed unknots, since we can merge them with other loops. 
\begin{equation}\label{eqn_23}
\includegraphics[valign=c,scale=.4]{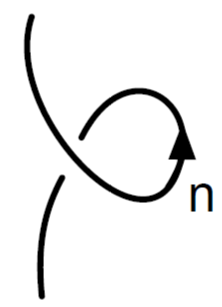} = \includegraphics[valign=c,scale=.4]{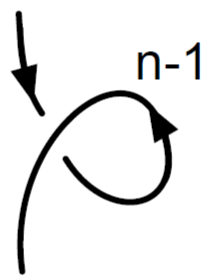} 
\quad \implies \quad
A \includegraphics[valign=c,scale=.4]{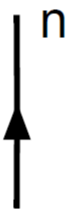} + A^{-1} \includegraphics[valign=c,scale=.4]{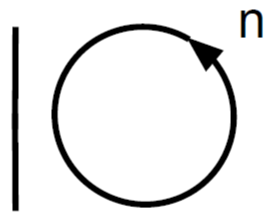} = A^{-1} \includegraphics[valign=c,scale=.4]{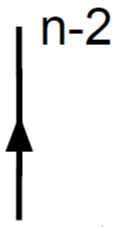} + A \includegraphics[valign=c,scale=.4]{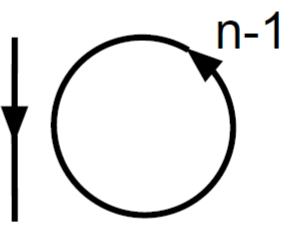}.
\end{equation}

Equation \eqref{eqn_23} implies Proposition \ref{prop_2.3}.
\begin{prop}[\cite{KBSM_S1}]\label{prop_2.3}
The skein module $\mathcal{S}(\Sigma\times S^1)$ is spanned by arrowed multicurves containing no trivial component, and by the arrowed multicurves consisting of just one arrowed unknot with some number of boundary parallel arrowed curves.
\end{prop}

\begin{defn}[Dual graph \cite{KBSM_S1}]
Let $\gamma\subset \Sigma$ be an arrowed multicurve. 
Let $c$ be the multicurve consisting of one copy of each isotopy class of separating essential loop in $\gamma$. Let $V$ be the set of connected components of $\Sigma-c$. For $v\in V$, denote by $\Sigma(v)\subset \Sigma$ the corresponding connected component of $\Sigma-c$. Two distinct vertices share an edge $(v_1, v_2) \in E$ if the subsurfaces $\Sigma(v_1)$ and $\Sigma(v_2)$ have a common boundary component. Define the \textbf{dual graph} of $\gamma$ to be the graph $\Gamma(\gamma)=(V,E)$. 
\end{defn}

\subsection{Relations between skeins}\label{section_relations_pants}
We now study some operations among arrowed multicurves in $\mathcal{S}(\Sigma\times S^1)$ that change the number of arrows in a controlled way. Although one can observe that all relations happen on a three-holed sphere, we write them separately for didactical purposes. 

In practice, a vertical strand will be part of a concentric circle. Lemma \ref{lem_popping_arrows} states that we can `pop-out' the arrows from a loop with the desired sign (of $y$ and $x$) without increasing the number of arrows in the diagrams. Lemma \ref{lem_flip_unknot} states that we can change the sign of the arrows in an unknot at the expense of adding skeins with fewer arrows. Lemma \ref{lem_merging_unknots} allows us to `break' and `merge' the arrows in between two unknots. Lemma \ref{lem_pushing_arrows} lets us pass arrows between parallel (or nested), and Lemma \ref{lem_crossing_borders} is an explicit case of Equation \eqref{eqn_23}.
The symbol $\includegraphics[scale=.65]{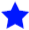}$ in Lemmas \ref{lem_merging_unknots} and \ref{lem_pushing_arrows} will correspond to any subsurface of surface $\Sigma$. In practice, $\includegraphics[scale=.65]{blue_star.png}$ will correspond to a boundary component of $\Sigma$ or an exceptional fiber in Section \ref{section_SFS}.

\begin{lem} \label{lem_popping_arrows}
For any $m\in \Z-\{0\}$, 
\begin{enumerate}[label=(\roman*)] 
\item 
$ \includegraphics[valign=c,scale=.55]{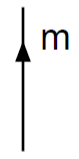} \in \Z[A^\pm] \bigg\{ \quad\includegraphics[valign=c,scale=.55]{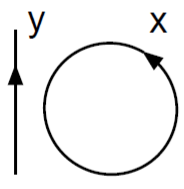}\quad : mx\geq 0, y\in \{0,1\}, y+|x|\leq |m|\bigg\}$.
\item
$ \includegraphics[valign=c,scale=.55]{fig_14_m.png} \in \Z[A^\pm] \bigg\{ \quad\includegraphics[valign=c,scale=.55]{fig_14_yx.png}\quad : mx\geq 0, y\in \{0,-1\}, |y|+|x|\leq |m|\bigg\}$.
\end{enumerate}
\end{lem}
\begin{proof}
Add one arrow pointing upwards at the top end of the arcs in Equation \eqref{eqn_23} and set $n=m-1$. We obtain the following equation
\begin{equation}\label{eqn_23+}
A \includegraphics[valign=c,scale=.5]{fig_14_m.png} + A^{-1} \includegraphics[valign=c,scale=.5]{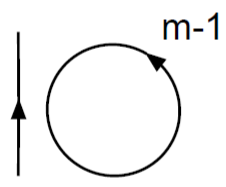} = A^{-1} \includegraphics[valign=c,scale=.5]{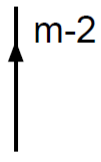} + A \includegraphics[valign=c,scale=.5]{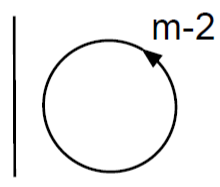}.
\end{equation}
If $m>0$, we can solve for \includegraphics[valign=c,scale=.5]{fig_14_m.png} and use it inductively to show Part (i). If $m<0$, we can instead solve for \includegraphics[valign=c,scale=.5]{fig_14_m-2.png} and set $m'=m-2$. This new equation can be use to prove Part (i) for $m'<0$. \\
Part (ii) is similar. Start with Equation \eqref{eqn_23} with $m=n$ and solve for \includegraphics[valign=c,scale=.5]{fig_14_m.png} to prove Part (ii) for $m>0$. If $m<0$, set $m=n-2$ in Equation \eqref{eqn_23}.
\end{proof}

\begin{lem}[Proposition 4.2 of \cite{KBSM_S1}]\label{lem_flip_unknot}
Let $S_k$ be an unknot in $\Sigma$ with $k\in \Z$ arrows oriented counterclockwise. The following holds for $n\geq 1$, 
\begin{enumerate}[label=(\roman*)] 
\item $S_1 = A^6 S_{-1}$
\item $S_{-n} = A^{-(2n+4)} S_{n}$ modulo $\Q(A)\{ S_0, \dots, S_{n-1}\}$. 
\item $S_{n} = A^{2n+4} S_{-n}$ modulo $\Q(A)\{ S_{-(n-1)}, \dots, S_{0}\}$. 
\end{enumerate}
\end{lem}

\begin{lem}\label{lem_merging_unknots}
Let $a, b \in \Z$ with $ab>0$. Then  
\begin{enumerate}[label=(\roman*)] 	
\item $\includegraphics[valign=c,scale=.5]{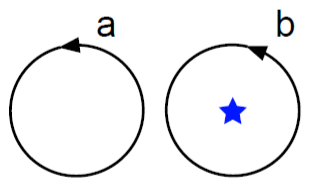} +A^{\frac{2a}{|a|}} \includegraphics[valign=c,scale=.5]{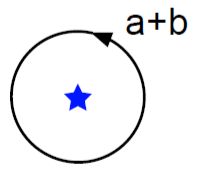} \in \Z[A^{\pm1}] \big\{  \includegraphics[valign=c,scale=.5]{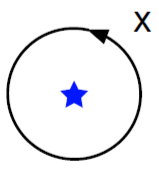} : 0\leq ax, 0\leq |x|<|a|+|b|\big\}$.
\item $\includegraphics[valign=c,scale=.5]{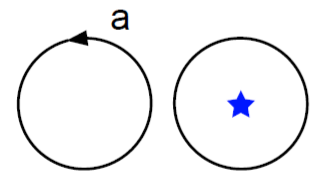} \in \Z[A^{\pm1}] \big\{  \includegraphics[valign=c,scale=.5]{fig_16_x.png} : 0\leq |x|\leq|a|\big\}$.
\end{enumerate}
\end{lem}
\begin{proof}
Suppose first that $a, b>0$. Using R4 we obtain $\includegraphics[valign=c,scale=.5]{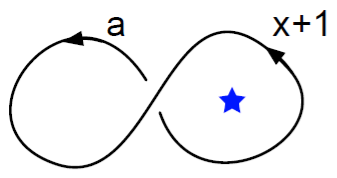} = \includegraphics[valign=c,scale=.5]{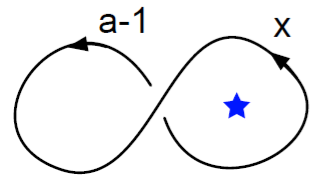}$. Thus, 
\begin{equation}
\includegraphics[valign=c,scale=.5]{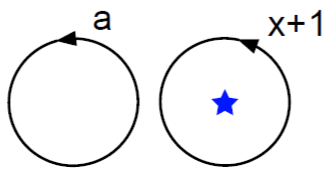}  +A^2 \includegraphics[valign=c,scale=.5]{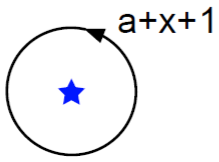} = A^2 \includegraphics[valign=c,scale=.5]{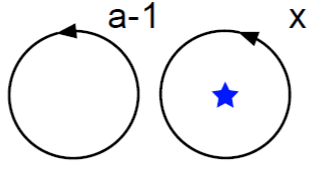} + \includegraphics[valign=c,scale=.5]{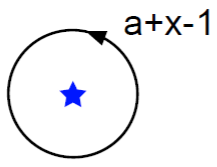}.
\end{equation}
By setting $x=0$, the statement follows for $b=1$ and all $a\geq 1$. For general $b\geq 1$, we proceed by induction on $a\geq 1$ setting $x=b$ in the equation above. 
The proof of case $ab<0$ uses the equation above after the change of variable $a=-a'+1$ and $x=-x'$. Part (ii) follows from Equation \eqref{eqn_23} with $n=a$.
\end{proof}

\begin{lem} \label{lem_pushing_arrows}
For all $a,b\in \Z$, 
\begin{align*}
&\text{(i) }
 \includegraphics[valign=c,scale=.35]{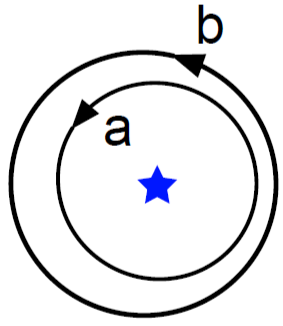} = A^2 \includegraphics[valign=c,scale=.35]{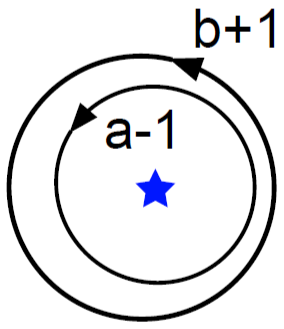} + \includegraphics[valign=c,scale=.45]{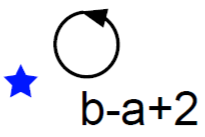} - A^2 \includegraphics[valign=c,scale=.45]{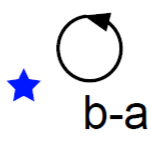}.\\
&\text{(ii) }
 \includegraphics[valign=c,scale=.35]{pushing_p_a_b.png} = A^{-2} \includegraphics[valign=c,scale=.35]{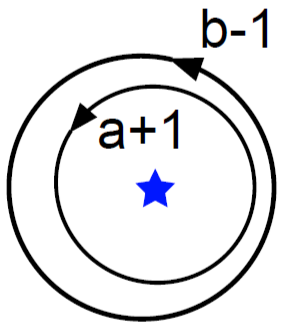} + \includegraphics[valign=c,scale=.45]{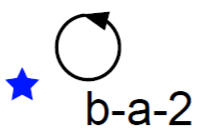} - A^{-2} \includegraphics[valign=c,scale=.45]{pushing_p_b-a.png}. \\
&\text{(iii) }
\includegraphics[valign=c,scale=.35]{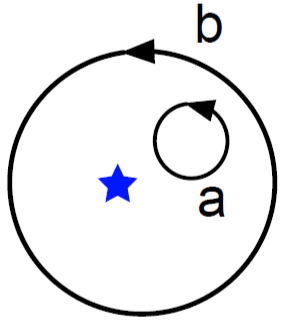} = A^2 \includegraphics[valign=c,scale=.35]{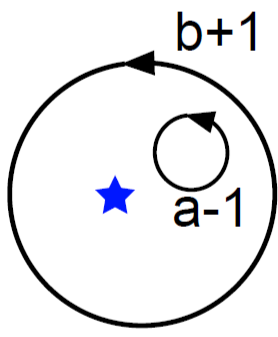} + \includegraphics[valign=c,scale=.35]{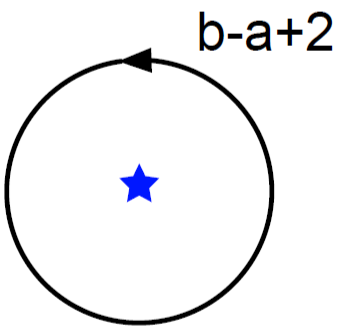} - A^2 \includegraphics[valign=c,scale=.35]{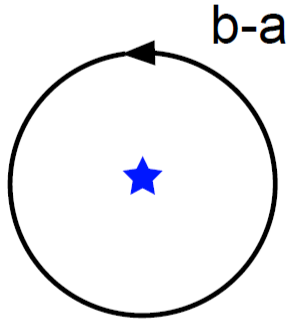}.\\
&\text{(iv) }
\includegraphics[valign=c,scale=.35]{pushing_n_a_b.png} = A^{-2} \includegraphics[valign=c,scale=.35]{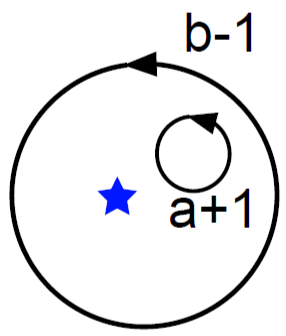} + \includegraphics[valign=c,scale=.35]{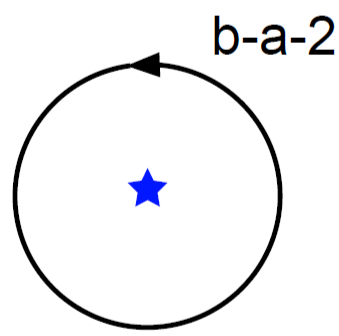} - A^{-2} \includegraphics[valign=c,scale=.35]{pushing_n_b-a.png}.\\
\end{align*}
\end{lem}
\begin{proof}
One can use (K1) on the LHS of each equation to create a new croossing. The result follows from (R4). 
\end{proof}

\begin{lem} \label{lem_crossing_borders}
$\quad$
\[ \includegraphics[valign=c,scale=.55]{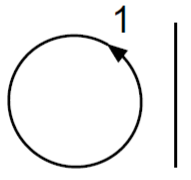} = -A^{4} \includegraphics[valign=c,scale=.55]{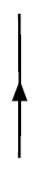} - A^{2} \includegraphics[valign=c,scale=.55]{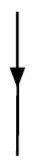}.\]
\end{lem}
\begin{proof}
Rotate Equation \eqref{eqn_23} by 180 degrees and set $n=1$.
\end{proof}


\section{Planar case}\label{section_planar_case}
Fix a planar subsurface $\Sigma' \subset \Sigma$ with at least 4 boundary components. The goal of this section is to prove Proposition \ref{thm_planar_case} which states that $\mathcal{S}(\Sigma'\times S^1)$ is generated by arrowed diagrams with $\partial$-parallel arrowed curves only. In particular, the dimension of $\mathcal{S}(\Sigma'\times S^1)$ as a module over its boundary is one; generated by the empty link.

We will study diagrams in linear pants decompositions. These are pants decompositions for $\Sigma'$ with dual graph isomorphic to a line. See Figure \ref{fig_linear_pants} for a concrete picture. Linear decompositions have $N=|\chi(\Sigma')|\geq 2$ pairs of pants. By fixing a linear pants decomposition, there is a well-defined notion of left and right ends of $\Sigma'$. We denote the specific curves of a linear pants decomposition as in Figure \ref{fig_linear_pants}. 
We think of such decomposition as the planar analogues for the sausague decompositions of positive genus surfaces in \cite{KBSM_S1}. 
\begin{figure}[h]
\centering
\includegraphics[scale=.5]{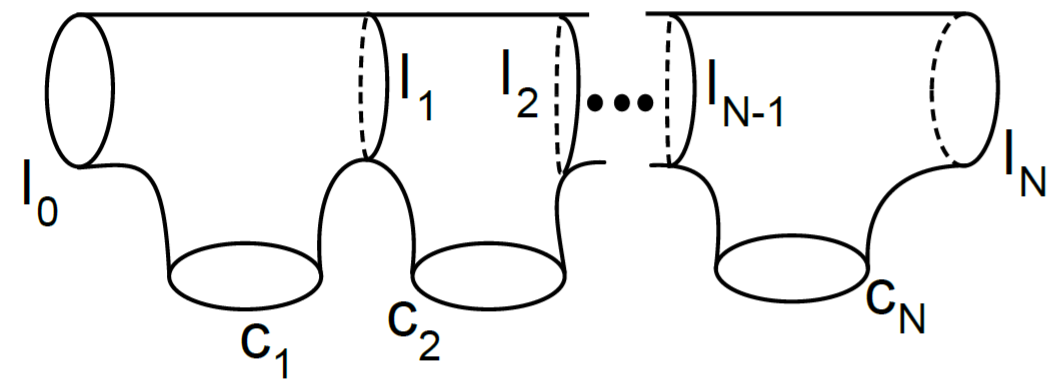}
\caption{Linear pants decomposition for spheres with holes.}
\label{fig_linear_pants}
\end{figure}

The \textbf{main idea of Proposition \ref{thm_planar_case}} is to `push' loops parallel to $l_i$ in a linear pants decomposition towards the boundary of $\partial \Sigma$ in both directions. We do this with the help of the $\Delta$-maps from Definition \ref{defn_delta}; $\Delta_+$ `pushes' loops towards the left and $\Delta_-$ towards the right (see Lemma \ref{equation_1_planar}). This idea is based on Section 3.3 of \cite{KBSM_S1} where the authors concluded a version of Proposition \ref{thm_planar_case} for closed surfaces. The following definition helps us to keep track of the arrowed curves in the boundary. 

\begin{defn}[Diagrams in linear pants decompositions]\label{defn_linear_pants}
Fix a linear pants decomposition of $\Sigma'$ and integers $m\geq 0$, $k_0 \in \{1, \dots, N-1\}$. For each $k\in \{0,\dots, N\} - \{k_0\}$, $a\in \Z$, and $v\in \{0,1,\emptyset\}^N$, we define the arrowed multicurves $D^k_{a,v}$ in $\Sigma'$ as follows: $D^{k}_{a,v}$ has one copy of $l_k$ with $a$ arrows, $m$ copies of $l_{k_0}$with no arrows, and one copy of $c_i$ with $v_i$ arrows if $v_i=0,1$ and no curve $c_i$ if $v_i=\emptyset$. Notice that the positive direction of the arrows of the curves $c_i$ depends on the (left/right) position of $c_i$ with respect to $l_{k_0}$. 
If $k=k_0$, we define ${}_lD^{k_0}_{a,v}$ (resp. ${}_rD^{k_0}_{a,v}$) as before with the condition that the left-most (resp. right-most) copy of $l_{k_0}$ contains $a$ arrows. 
\end{defn}

\begin{figure}[h]
\centering
\includegraphics[scale=.65]{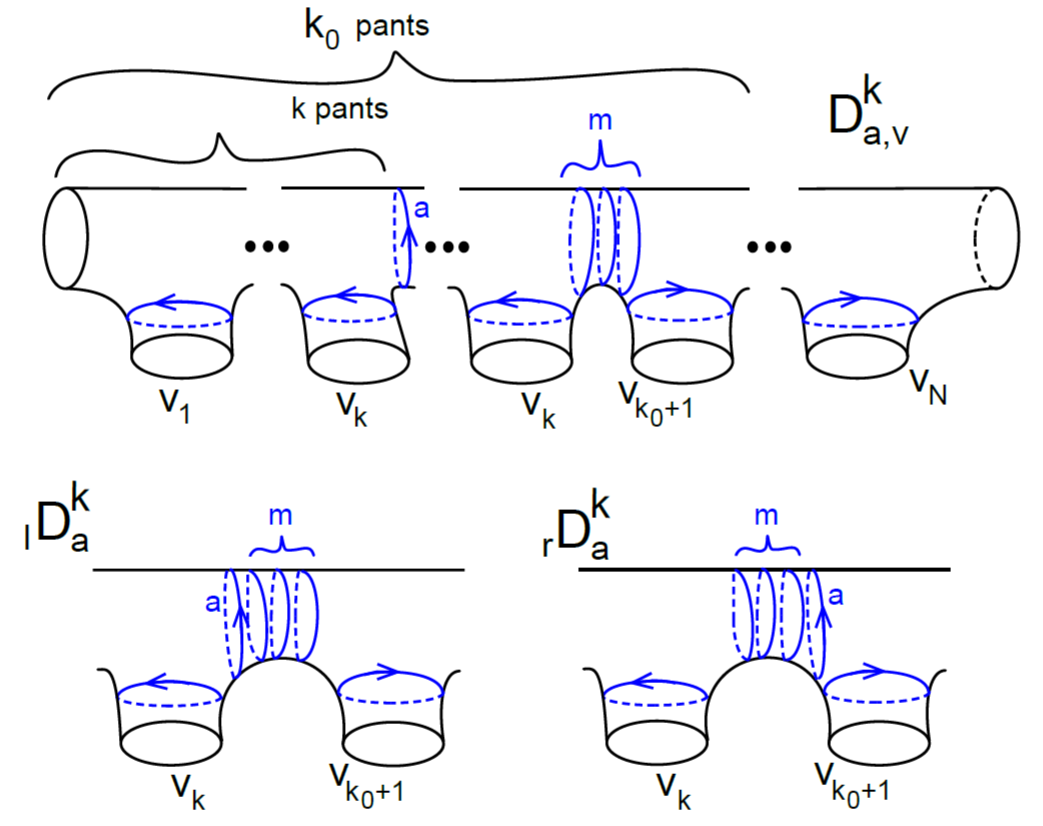}
\caption{Definition of $D^k_{a,v}$, ${}_lD^{k_0}_{a,v}$ and ${}_rD^{k_0}_{a,v}$.}
\label{fig_Dav}
\end{figure}

\begin{lem}[Lemma 3.11 of \cite{KBSM_S1}]\label{lem_3.11}
The following holds for any two parallel curves, 
\[ A^{-1}\includegraphics[valign=c,scale=.4]{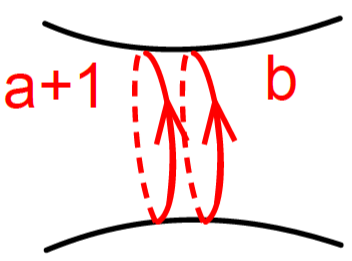} + A \includegraphics[valign=c,scale=.4]{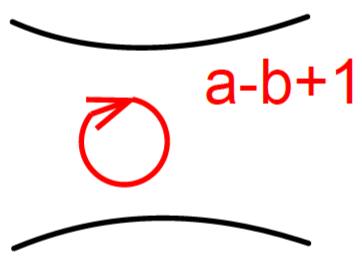} =  A \includegraphics[valign=c,scale=.4]{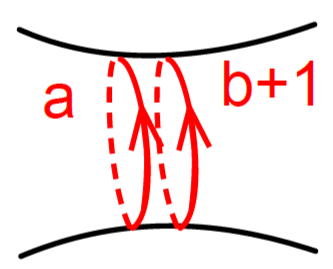} + A^{-1} \includegraphics[valign=c,scale=.4]{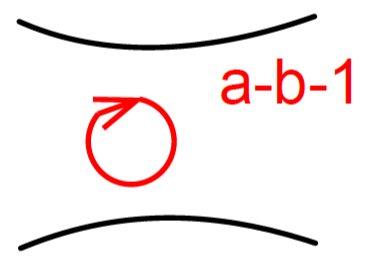}.\]

In particular, for any $a \in \Z$, $m\geq 0$, and $v\in \{0,1,\emptyset\}^N$, we have 
\[ {}_lD^{k_0}_{a,v} \cong A^{2ma} {}_rD^{k_0}_{a,v}, \]
modulo $\Z[A^{\pm1}]$-linear combinations of diagrams with fewer non-trivial loops. 
\end{lem}

Lemma \ref{lem_3.13} allows us to change the location of the $a$ arrows in the diagram $D$ at the expense of changing the vector $v$. Its proof follows from Proposition 3.5 of \cite{KBSM_S1}. 

\begin{lem}\label{lem_3.13}
The following equations hold for $k\in \{1,\dots, N-1\}$. 
\begin{enumerate}[label=(\roman*)] 
\item If $k>k_0$, then 
\[ A D^k_{a,(\dots, v_k, \emptyset ,\dots)} - A^{-1}D^k_{a+2,(\dots, v_k, \emptyset, \dots)} = A D^{k+1}_{a+1,(\dots, v_k, 1, \emptyset,\dots)} - A D^{k+1}_{a,(\dots, v_k, 0, \emptyset,\dots)}. \]
\item If $k=k_0$, then 
\[ A {}_r D^{k_0}_{a,(\dots, v_k, \emptyset,\dots)} - A^{-1} {}_r D^{k_0}_{a+2,(\dots, v_k, \emptyset,\dots)} = A D^{k_0+1}_{a+1,(\dots, v_k, 1, \emptyset,\dots)} - A^{-1} D^{k_0+1}_{a,(\dots, v_k, 0, \emptyset,\dots)}. \]
\item If $k = k_0$, then 
\[  A {}_l D^{k_0}_{a+2,(\dots, \emptyset, v_{k_0}, \dots)} - A^{-1} {}_l D^{k_0}_{a,(\dots, \emptyset, v_{k_0}, \dots)} = A D^{k_0-1}_{a,(\dots,\emptyset, 0, v_{k_0}, \dots)} - A^{-1} D^{k_0-1}_{a+1,(\dots,\emptyset, 1, v_{k_0}, \dots)}. \]
\item If $k < k_0$, then 
\[  A D^{k}_{a+2,(\dots, \emptyset, v_{k}, \dots)} - A^{-1} D^{k}_{a,(\dots, \emptyset, v_{k}, \dots)} = A D^{k-1}_{a,(\dots,\emptyset, 0, v_{k}, \dots)} - A^{-1} D^{k-1}_{a+1,(\dots,\emptyset, 1, v_{k}, \dots)}. \]
\end{enumerate}
\end{lem}

\begin{defn}[$\Delta$-maps]\label{defn_delta}
Following \cite{KBSM_S1}, let $V^{\partial \Sigma'}$ be the subspace of $\mathcal{S}(\Sigma'\times S^1)$ generated by arrowed diagrams with trivial loops and boundary parallel curves in $\Sigma'$. Consider $V$ to be the formal vector space over $\Q(A)$ spanned by the diagrams $D^k_{a,v}$, ${}_l D^{k_0}_{a,v}$ and ${}_r D^{k_0}_{a,v}$ for all $a\in \Z$, $v\in \{0,1,\emptyset\}^{N}$ and $k\in \{0,\dots, N\} \setminus \{k_0\}$. Define the linear map $s: V\ra V$ given by $s(D^k_{a,v}) = D^k_{a+2,v}$ (similarly for ${}_l D^{k_0}_{a,v}$ and ${}_r D^{k_0}_{a,v}$). Define the maps $\Delta_-, \Delta_+$, and $\Delta_{+, m}$ by
\[ \Delta_- = A-A^{-1}s, \quad \Delta_+=As-A^{-1}, \quad \Delta_{+,m} = A^{4m+1}s - A^{-1}.\]
\end{defn}
Combinations of $\Delta$-maps, together with Lemmas \ref{equation_1_planar} and \ref{lem_3.14}, will show that $V\subset V^{\partial \Sigma'}$.

\begin{lem}\label{equation_1_planar}
Let $o(e)$ and $z(e)$ be the number of ones and zeros of a vector $e\in \{0,1\}^n$. 
\begin{enumerate}[label=(\roman*)] 
\item The following equation holds for all $1\leq n \leq k_0$. 
\begin{equation}\label{eqn_1_planar}
\Delta^n_+ \left( {}_l D^{k_0}_{a, (\dots, \emptyset, \dots)} \right)= \sum_{e \in \{0,1\}^n} (-1)^{o(e)} A^{z(e) - o(e)} D^{k_0 - n}_{a+o(e), (\dots, \emptyset, e, \emptyset, \dots)}, \end{equation}
where $e=(e_1, \dots, e_n)$ is located so that $v_{k_0} = e_n$. 
\item The following equation holds for all $1\leq n \leq N-k_0$. 
\begin{equation}\label{eqn_2_planar}
\Delta^n_- \left( {}_r D^{k_0}_{a, (\dots, \emptyset, \dots)} \right)= \sum_{e \in \{0,1\}^n} (-1)^{z(e)} A^{o(e) - z(e)} D^{k_0 + n}_{a+o(e), (\dots, \emptyset, e, \emptyset, \dots)}, 
\end{equation}
where $e=(e_1, \dots, e_n)$ is located so that $v_{k_0} = e_1$. 
\end{enumerate}
\end{lem}
\begin{proof} 
We now prove Equation \eqref{eqn_1_planar}. The proof of Equation \eqref{eqn_2_planar} is symmetric and it is left to the reader. 
Lemma \ref{lem_3.13} with $k=k_0$ is the statement of case $n=1$. We proceed by induction on $n$ and suppose that Equation \eqref{eqn_1_planar} holds for some $1\leq n \leq k_0-1$. Using Lemma \ref{lem_3.13} with $k<k_0$, we show the inductive step as follows, 

\begin{align*} 
\Delta^{n+1}_{+}\left( {}_l D^{k_0}_{a, \emptyset}\right) = &  \left(As - A^{-1}\right) \circ \Delta^{n}_{+} \left( {}_l D^{k_0}_{a, \emptyset}\right) \\ 
 = & \sum_{e \in \{0,1\}^n}  (-1)^{o(e)} A^{1+z(e) - o(e)}D^{k_0 - n}_{a+2+o(e), (\dots, \emptyset, e, \emptyset, \dots)} \\ & \quad \quad \quad \quad \quad \quad \quad \quad \quad  \quad \quad  \quad -(-1)^{o(e)} A^{-1 + z(e) -o(e)} D^{k_0 - n}_{a+o(e), (\dots, \emptyset, e, \emptyset, \dots)}  \\
 = & \sum_{e\in \{0,1\}^n} (-1)^{o(e)} A^{z(e) -o(e)} \left[ AD^{k_0 - n}_{a+o(e)+2, (\dots, \emptyset, e, \emptyset, \dots)} - A^{-1} D^{k_0 - n}_{a+o(e) , (\dots, \emptyset, e, \emptyset, \dots)} \right] \\ 
 = & \sum_{e\in \{0,1\}^n} (-1)^{o(e)} A^{z(e) -o(e)} \left[ AD^{k_0 - n-1}_{a+o(e), (\dots, \emptyset,0, e, \emptyset, \dots)} - A^{-1} D^{k_0 - n-1}_{a+o(e) +1, (\dots, \emptyset,1, e, \emptyset, \dots)} \right] \\ 
 = & \sum_{e\in \{0,1\}^n} (-1)^{o(e)} A^{1+z(e) - o(e)}D^{k_0 - n-1}_{a+o(e), (\dots, \emptyset,0, e, \emptyset, \dots)} \\ & \quad \quad \quad \quad \quad \quad \quad \quad \quad \quad \quad  \quad +(-1)^{o(e)+1} A^{z(e) -o(e)-1} D^{k_0 - n-1}_{a+o(e)+1, (\dots, \emptyset,1, e, \emptyset, \dots)}  \\
 = & \sum_{e \in \{0,1\}^{n+1}} (-1)^{o(e)} A^{z(e) - o(e)} D^{k_0 - (n+1)}_{a+o(e), (\dots, \emptyset, e, \emptyset, \dots)}.
\end{align*}
\end{proof}

\begin{lem} \label{lem_3.14}
For any $a \in \Z$, we have $\Delta^{k_0}_{+} \left({}_l D^{k_0}_{a,\emptyset}\right), \Delta^{N-k_0}_{-} \left({}_r D^{k_0}_{a,\emptyset}\right)\in V^{\partial \Sigma'}$. Furthermore, $\Delta^{k_0}_{+} \left({}_l D^{k_0}_{a,\emptyset}\right)$ is a linear combination of elements of the form $D^{0}_{a',(v_1,\dots, v_{k_0},\emptyset,\dots, \emptyset)}$ and $\Delta^{N-k_0}_{-} \left({}_r D^{k_0}_{a,\emptyset}\right)$ is a sum of elements $D^{N}_{a',(\emptyset,\dots, \emptyset, v_{k_0+1}, \dots, v_N)}$.
\end{lem} 

\begin{proof} 
Setting $n=k_0$ in Equation \eqref{eqn_1_planar} yields the condition $\Delta^{k_0}_{+} \left({ }_l D^{k_0}_{a,\emptyset}\right)\in V^{\partial \Sigma'}$ and the first half of the statement. 
The second conclusion $\Delta^{N-k_0}_{-} \left({ }_r D^{k_0}_{a,\emptyset}\right)\in V^{\partial \Sigma'}$ follows by setting $n=N-k_0$ in Equation \eqref{eqn_2_planar}.
\end{proof}

\begin{prop}\label{prop_3.12} 
${ }_l D^{k_0}_{a,v}$ and ${ }_r D^{k_0}_{a,v}$ lie in $V^{\partial \Sigma'}$ for any $a\in \Z$ and $v\in \{0,1,\emptyset\}^N$.
\end{prop} 

\begin{proof}
By pushing the boundary parallel curves `outside' $\Sigma'$, it is enough to show the proposition for $v= \emptyset$. 
Using Lemma \ref{lem_3.11}, modulo arrowed multicurves with fewer non-trivial loops, we get that 
$ s\left( { }_l D^{k_0}_{a,\emptyset}\right) = { }_l D^{k_0}_{a+2,\emptyset} \cong A^{2m(a+2)} { }_r D^{k_0}_{a+2,\emptyset}$. 
Thus, 
\begin{align*}
\Delta_{+} \left( { }_l D^{k_0}_{a,\emptyset}\right) = & As\left( { }_l D^{k_0}_{a,\emptyset} \right) - A^{-1} { }_l D^{k_0}_{a,\emptyset} \\
\cong & A^{4m+1} A^{2ma} { }_r D^{k_0}_{a+2,\emptyset} - A^{-1} A^{2ma} { }_r D^{k_0}_{a,\emptyset}\\
= & A^{2ma} \left[ A^{4m+1}s - A^{-1}\right] \left( { }_r D^{k_0}_{a,\emptyset}\right) \\
= & A^{2ma} \Delta_{+,m} \left( { }_r D^{k_0}_{a,\emptyset}\right).
\end{align*} 
Hence, up to sums of curves with less non-trivial loops in $\Sigma'$, Lemma \ref{lem_3.14} implies \[\Delta^{k_0}_{+,m} \left( { }_r D^{k_0}_{a,\emptyset}\right),  \Delta^{N-k_0}_{-} \left( { }_r D^{k_0}_{a,\emptyset}\right) \in V^{\partial\Sigma'}.\]
Finally, observe that 
$ A^{-1} \Delta_{+,m} + A^{4m+1} \Delta_- = \left( A^{4m+2} - A^{-2} \right) Id_V$. 
This yields 
\[{ }_r D^{k_0}_{a,\emptyset} = Id_{V}^{N} \left( { }_r D^{k_0}_{a,\emptyset}\right) = \frac{1}{\left( A^{4m+2} - A^{-2} \right)^N }  \left(  A^{-1} \Delta_{+,m} + A^{4m+1} \Delta_- \right)^N \left( { }_r D^{k_0}_{a,\emptyset}\right).\]
The result follows since $\Delta^{k_0}_{+,m} \left( { }_r D^{k_0}_{a,\emptyset}\right)$ and $\Delta^{N-k_0}_{-} \left( { }_r D^{k_0}_{a,\emptyset}\right)$ are both elements of $V^{\partial \Sigma'}$. 
\end{proof}

We are ready to describe an explicit generating set for $\S(\Sigma\times S^1)$ for any planar surface $\Sigma$. 
\begin{prop} \label{thm_planar_case}
Let $\Sigma$ be a $N$-holed sphere with $N\geq 1$. Then $\mathcal{S}(\Sigma \times S^1)$ is generated by arrowed unknots and $\partial$-parallel arrowed multicurves. 
\end{prop}

\begin{proof} 
Proposition \ref{thm_planar_case} is equivalent to the statement that $\mathcal{S}(\Sigma \times S^1)$ is generated by arrowed multicurves with dual graphs isomorphic to a point. 
Let $\gamma$ be an arrowed multicurve in $\Sigma$ with $\Gamma(\gamma)\neq \{pt\}$. Let $e =(v_1,v_2)$ be a fixed edge of $\Gamma(\gamma)$, and let $\Sigma' \subset \Sigma$ be the subsurface $\Sigma(v_1) \cup \Sigma(v_2)$. 
By Lemma \ref{lem_3.11}, up to curves of smaller degree, we can arrange the arrows in the loops corresponding to $e$ so that only one loop (the closest to $\Sigma(v_2)$) may have arrows. 
By construction, $\gamma \cap \Sigma'$ has one isotopy class of separating non $\partial$-parallel curve in $\Sigma'$. Thus, there exists a linear pants decomposition for $\Sigma'$ and integers $a\in \Z$, $k_0 \in \{1, \dots, |\chi(\Sigma')|-1\}$ so that $\gamma \cong {}_r D^{k_0}_{a,\emptyset}$ (we focus on the non $\partial$-parallel components of $\gamma\cap \Sigma'$). Proposition \ref{prop_3.12} states that ${}_r D^{k_0}_{a,\emptyset}\in V^{\partial \Sigma'}$. 
Therefore, $\gamma$ is a $\Q(a)$-linear combination of arrowed multicurves with dual graphs isomorphic to $\Gamma(\gamma)/e$; with fewer vertices than $\Gamma(\gamma)$. 
\end{proof} 


\section{Non-planar case}\label{section_non_planar_case}

This section further exploits the proofs in \cite{KBSM_S1} to give a finiteness result for $\mathcal{S}(\Sigma \times S^1)$ for all orientable surfaces with boundary (Proposition \ref{thm_boundary_case}). Throughout this section, $\Sigma$ will be a compact orientable surface of genus $g>0$ with $N>0$ boundary components.

\subsection{Properties of stable multicurves}
\begin{defn}[Complexity]
Let $\gamma$ be an arrowed multicurve. Denote by $n$ the number of non-separating circles of $\gamma$, $m$ the number of non-trivial non $\partial$-parallel separating circles in $\gamma$, and $b$ the number of vertices in the dual graph of $\gamma$ intersecting $\partial\Sigma$. We define the complexity of a multicurve $\gamma$ as $\left( b, n+2m, n+m\right)$ and order them with the lexicographic order.  An arrowed multicurve is said to be \textbf{stable} if it is not a linear combination of diagrams with lower complexity.
\end{defn}

Proposition \ref{prop_3.7}, Proposition \ref{prop_3.8}, and Lemma \ref{lem_3.9} restate properties of stable curves from \cite{KBSM_S1}. Fix a stable arrowed multicurve $\gamma$ in $\Sigma$. 

\begin{prop}[Proposition 3.7 of \cite{KBSM_S1}] \label{prop_3.7}
Let $\Sigma' = \Sigma(v)$ be a vertex of $\Gamma$ with $|\partial \Sigma'|\geq 1$ and $g(\Sigma')\geq 1$. Then $\gamma \cap \Sigma'$ contains at most one non-separating curve. 
\end{prop}

\begin{prop}[From proof of Proposition 3.8 of \cite{KBSM_S1}] \label{prop_3.8}
If $e=(v,v')$ is an edge of $\Gamma$ with $g(v') \geq 1$, then the valence of $v$ is at most two. 
\end{prop} 

\begin{lem} [Lemma 3.9 of \cite{KBSM_S1}] \label{lem_3.9}
For a vertex $v$ with $g(v)\geq 1$ and valence two, $\gamma\cap \Sigma(v)$ contains no non-separating curves.
\end{lem}

Now, Proposition \ref{prop_stable_line} shows that stable arrowed multicurves satisfy $b(\gamma)= 1$. 

\begin{prop}\label{prop_stable_line}
Stable arrowed multicurves have dual graphs isomorphic to lines. Moreover, they are $\Q(A)$-linear combinations of arrowed unknots and the two types of arrowed multicurves depicted in Figure \ref{fig_two_types}.
\end{prop} 
\begin{figure}[h]
\centering
\includegraphics[scale=.55]{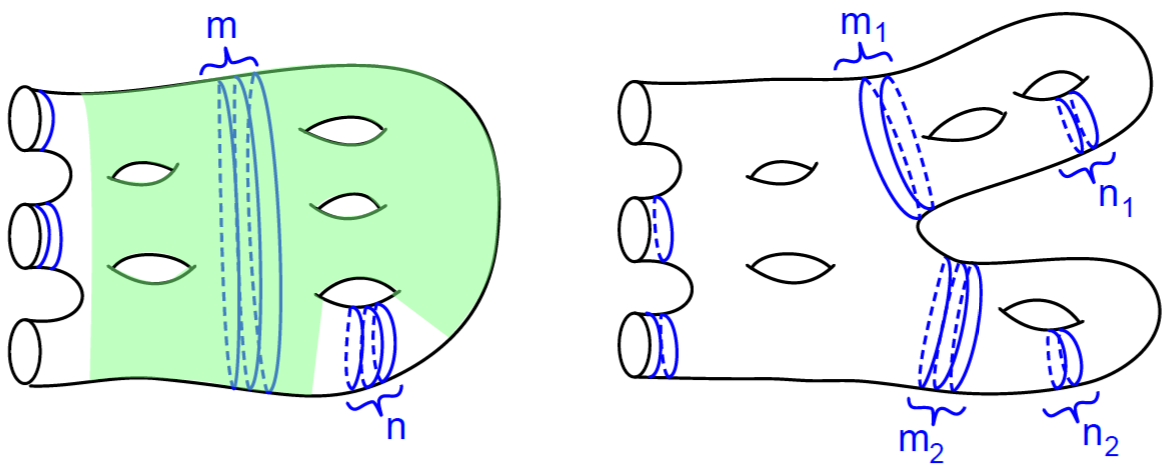}
\caption{Type 1 multicurves contain only one isotopy class of non-separating simple curve and type 2 at most two non-separating loops.}
\label{fig_two_types}
\end{figure}
\begin{proof} 
Suppose $b(\gamma)>1$; i.e., there exist two distinct vertices $v_1, v_2 \in \Gamma$ containing boundary components of $\Sigma$. We will show that $\gamma$ is not stable. 
There exists a path $P\subset \Sigma$ connecting $v_1$ and $v_2$. For each vertex $x \in P$, we define a subsurface $\Sigma'(x) \subset \Sigma(x)$ as follows: 
If $\Sigma(x)$ is planar, define $\Sigma'(x) := \Sigma(x)$. 
Suppose now that $g(x)\geq 1$ and $x\notin\{ v_1, v_2\}$. Proposition \ref{prop_3.7} states that $\gamma \cap \Sigma(x)$ contains at most one non-separating loop. Thus, we can find a planar surface $\Sigma'(x)\subset \Sigma(x)$ disjoint from the non-separating loop such that $\partial \Sigma'(x)$ contains the two boundaries of $\Sigma(x)$ participating in the path $P$ (see Figure \ref{fig_path}). 
Suppose now $g(x) \geq 1$ and $x=v_i$. Using Proposition \ref{prop_3.7} again, we can find a subsurface $\Sigma'(x)\subset \Sigma(x)$ with $\partial \Sigma'(x)$ containing the $\Sigma(x) \cap \partial \Sigma$ and the one loop of $\partial \Sigma(x)$ participating in the path $P$ (see Figure \ref{fig_path}). 
Define $\Sigma'\subset \Sigma$ to be the connected surface obtained by gluing the subsurfaces $\Sigma'(x)$ for all $x\in P$. Since $\Gamma$ is a tree, $\Sigma'$ must be planar. 

\begin{figure}[h]
\centering
\includegraphics[scale=.5]{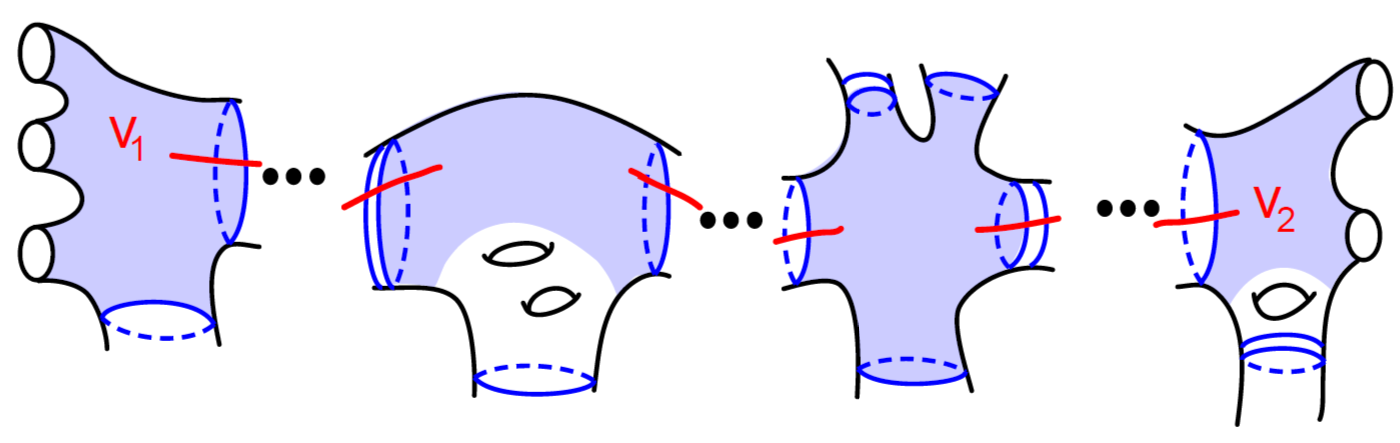}
\caption{Building the subsurface $\Sigma'$.}
\label{fig_path}
\end{figure}

By construction $\gamma \cap \Sigma'$ can be thought of as an element of $\mathcal{S}(\Sigma' \times S^1)$. Proposition \ref{thm_planar_case} states that $\gamma \cap \Sigma'$ can be written as $\Q(a)$-linear combination of arrowed diagrams with only trivial and $\partial$-parallel curves in $\Sigma'$. In particular, $\gamma$ can be written as a linear combination of arrowed diagrams $\gamma'$ in $\Sigma$ with $b(\gamma')<b(\gamma)$, and so $\gamma$ is not stable.

Let $\gamma$ be an stable arrowed multicurve. Since $b(\gamma)=1$, there is a unique vertex $x_0 \in \Gamma$ with $\partial \Sigma \subset \Sigma(x_0)$. Notice that any vertex $v\in \Gamma$ of valence two either has positive genus or is equal to $x_0$. This assertion, together with Proposition \ref{prop_3.8}, implies that $\Gamma$ is isomorphic to a line where every vertex different than $x_0$ has positive genus. 

The graph $\Gamma \setminus \{x_0\}$ is the disjoint union of at most two linear graphs $\Gamma_1$ and $\Gamma_2$; $\Gamma_i$ might be empty. For each $\Gamma_i\neq \emptyset$, the subsurface $\Sigma(\Gamma_i)$ is a surface of positive genus with one boundary component. 
If each $\Gamma_i$ has at most one vertex then $\gamma$ looks like curves in Figure \ref{fig_two_types} and the proposition follows. Suppose then that $\Gamma_i$ has two or more vertices and pick an edge $e$ of $\Gamma_i$. 
By Proposition \ref{prop_3.7} and Lemma \ref{lem_3.9}, $\gamma \cap \Sigma(e) $ contains at most one isotopy class of non-separating curves. Denote such curve by $\alpha$; observe that $\alpha$ is empty unless $e$ is has an endpoint on a leaf of $\Gamma$. Let $\Sigma''$ be the complement of an open neighborhood of $\alpha$ in $\Sigma(e)$. By construction, $\gamma \cap \Sigma''$ contains one isotopy class of non-trivial separating curves in $\Sigma''$. 
By Lemma 3.12 of \cite{KBSM_S1} we can `push' the separating arrowed loops in $\gamma \cap \Sigma''$ towards the boundary of $\Sigma''$. Thus, we can write $\gamma$ as a linear combination of diagrams with dual graph $\Gamma / e$. We can repeat this process until we obtain only summands with each $\Gamma_i$ having at most one vertex. 
\end{proof}

\begin{figure}[h]
\centering
\includegraphics[scale=.6]{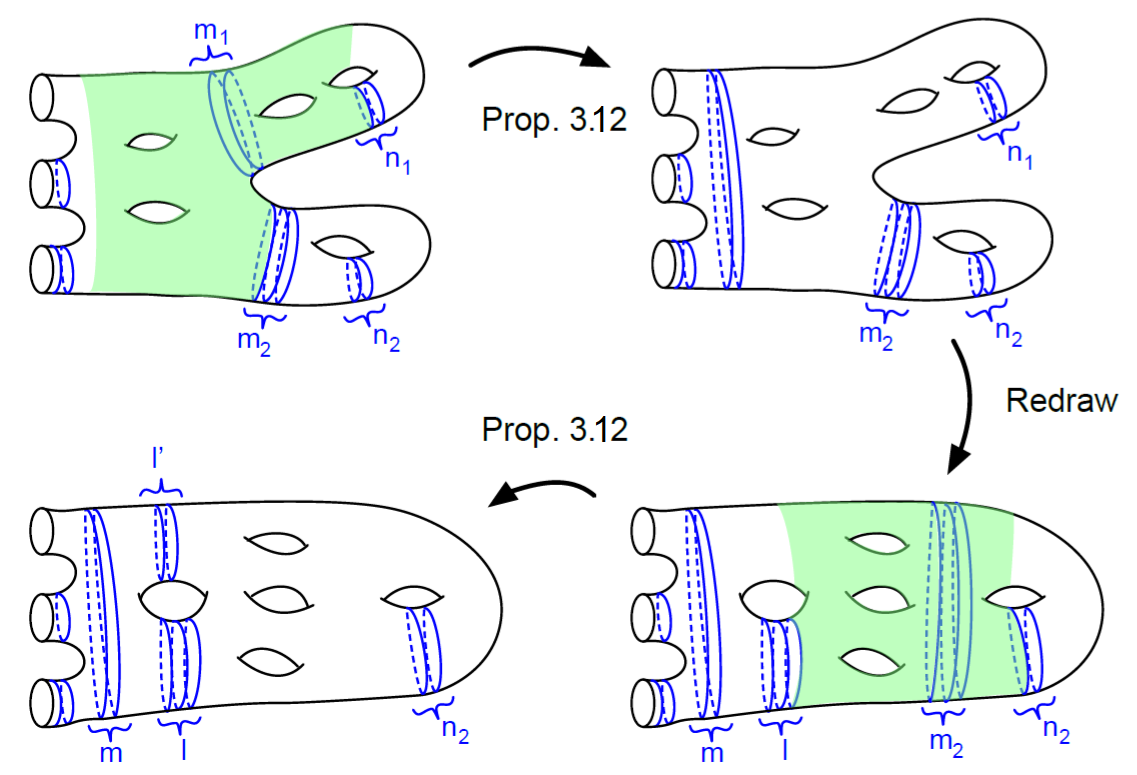}
\caption{One needs to apply Proposition 3.12 of \cite{KBSM_S1} twice for diagrams of type 2.}
\label{fig_last_step}
\end{figure}

\begin{prop} \label{prop_genus_case}
Let $\Sigma$ be an orientable surface of genus $g>0$ and $N>0$ boundary components. 
Then $\mathcal{S}(\Sigma\times S^1)$ is generated by arrowed unknots and arrowed multicurves with $\partial$-parallel components and at most one non-separating simple closed curve.
\end{prop}

\begin{proof} 
Using Proposition 3.12 of \cite{KBSM_S1} with $\Sigma'$ being the shaded surfaces in Figures \ref{fig_two_types} and \ref{fig_last_step}, we obtain that the $\mathcal{S}(\Sigma\times S^1)$ is generated by arrowed diagrams as in Figure \ref{fig_last_step} where $l+l'=n_1$ and $m,n_1, n_2 \geq 0$. Observe that, ignoring the $m$ curves, the $l$ and $l'$ curves are parallel. 
Also observe that, by Lemma \ref{lem_3.11}, we can still pass arrows among the $l$ and $l'$ curves modulo linear combinations of diagrams of the same type with lower $n_1$ but higher $m$. Thus, if we only focus on the complexity $n_1 + n_2$, we can follow the proof of Proposition 3.16 in \cite{KBSM_S1} and conclude that $\mathcal{S}(\Sigma\times S^1)$ is generated by arrowed diagrams with $n_1+ n_2\leq 1$. 

The rest of this proof focuses on making $m=0$. In order to do this, we combine techniques in Section \ref{section_planar_case} of this paper and Proposition 3.12 of \cite{KBSM_S1}.

\textbf{Case 1: $n_1+n_2=0$.} Fix $m\geq 0$. Let $c$ be a separating curve cutting $\Sigma$ into a sphere with $N+1$ holes and one connected surface of genus $g>0$ with one boundary component. The diagrams in this case contain only boundary parallel curves and copies of $c$. Define $V^{\partial \Sigma}_m$ to be the formal vector space defined by such pictures with at most $m$ parallel separating curves. For each $a\in \Z$, define the diagram ${}_r D_a$ (resp. ${}_l D_a$) to be given by $m+1$ copies of $c$, $m$ of which have no arrows and where the closest to the positive genus surface (resp. to the holed sphere) has $a$ arrows. By Lemma \ref{lem_3.11}, in order to conclude this case, we only need to check ${}_r D_a\in V^{\partial \Sigma}_m$. 

Define $\Delta_+$, $\Delta_-$ and $\Delta_{+,m}$ as in Section \ref{section_planar_case}. 
First, observe that Lemma \ref{lem_3.14} implies that $\Delta^{N-1}_+ \left( {}_l D_a \right) \in V^{\partial \Sigma}_m$. Using the computation in the proof of Proposition \ref{prop_3.12}, we conclude that  $\Delta^{N-1}_{+,m} \left( {}_r D_a \right) \in V^{\partial \Sigma}_m$. On the other hand, Lemma 3.14 of \cite{KBSM_S1} gives us $\Delta^{2g}_- \left( {}_r D_a \right)\in V^{\partial \Sigma}_m$. 
Hence, 
\[{}_r D_a = Id_{V}^{2g+N-1}\left( {}_r D_a\right) =\frac{1}{\left( A^{4m+2} - A^{-2} \right)^{2g+N-1} } \left(  A^{-1} \Delta_{+,m} + A^{4m+1} \Delta_- \right)^{2g+N-1}\left({}_r D_a\right) \in V^{\partial \Sigma}_m.\] 

\textbf{Case 2: $n_1+n_2 = 1$.} Fix $m\geq 0$. The diagrams in this case contain boundary parallel curves, some copies of $c$, and exactly one non-separating curve denoted by $\alpha$. Define $V^{\partial \Sigma}_m$ to be the formal vector space defined by such pictures with at most $m$ copies of $c$. For $a\in \Z$, define ${}_l D_a$, ${}_r D_a$ as in Case 1 with the addition of one copy of $\alpha$. In order to conclude this case, it is enough to show ${}_r D_a \in V^{\partial \Sigma}_m$.

Suppose that $\alpha$ has $x \in \Z$ arrows. For $a,b \in \Z$, define ${}_r E_{a,b}$ and ${}_l E_{a,b}$ to be $m$ copies of $c$ with no arrows and three copies of $\alpha$ with arrows arranged as in Figure \ref{fig_left_right}. 
We can define the map $s$ on the diagrams ${}_rE_{a,b}$ and ${}_l E_{a,b}$ by $s({}_* E_{a,b}) = {}_* E_{a+1,b+1}$. This way, the maps $\Delta_-$, $\Delta_+$, $\Delta_{+,m}$ are defined on the diagrams $D_a$ and $E_{a,b}$. Define $\Delta_{-,1}=A - A^{-3}s$. 
Using Lemma \ref{lem_3.11}, up to linear combinations of diagrams in $V^{\partial \Sigma}_m$, we obtain the following: 
\begin{align*}
 \Delta_{-}\left( {}_r E_{a,0}\right) = &A{}_r E_{a,0} - A^{-1} {}_rE_{a+1,1} \\ 
 \cong & A^{2x+1} {}_l E_{a,0} - A^{2(x-1)-1} {}_l E_{a+1,1} \\ 
 = & A^{2x} \left[ A {}_l E_{a,0} - A^{-3} {}_l E_{a+1,1}\right]\\ 
 = & A^{2x} \Delta_{-,1} \left( {}_l E_{a,0}\right).
\end{align*}
\begin{figure}[h]
\centering
\includegraphics[scale=.6]{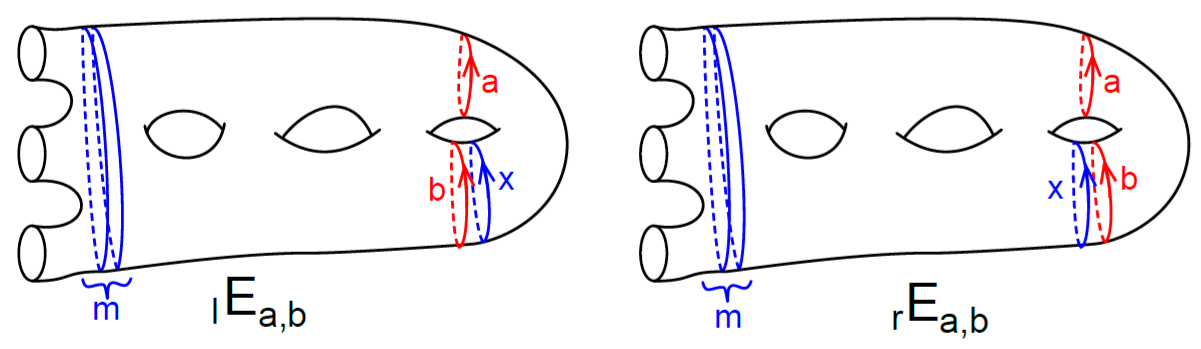}
\caption{${}_l E_{a,b,x}$ and ${}_r E_{a,b,x}$.}
\label{fig_left_right}
\end{figure}
Lemmas 3.13 and 3.14 of \cite{KBSM_S1} give us that $\Delta_- \left( {}_r E_{a,0} \right) \in V^{\partial \Sigma}_m$ and 
$\Delta^{2g-1}_- ({}_r D_a) = \Delta^{2g-1}_+ ({}_l E_{a,0})$. The first equation implies that $\Delta_{-,1} \left( {}_l E_{a,0}\right) \in V^{\partial \Sigma}_m$. This implication, together with the second equation and the fact that the $\Delta$-maps commute, lets us conclude that $\Delta_{-,1} \circ \Delta^{2g-1}_{-} ({}_r D_a) \in V^{\partial \Sigma}_m$.

Finally, notice that the argument in Case 1 implies that $\Delta^{N-1}_{+,m} \left( {}_r D_a \right) \in V^{\partial \Sigma}_m$. We also have the following relations between $\Delta$-maps. 
\[ A^{4m+3} \Delta_{-,1} + A^{-1} \Delta_{+,m} = \left( A^{4m+4} - A^{-2}\right) Id_{V} \] 
\[ A^{4m+1} \Delta_- + A^{-1} \Delta_{+,m} = \left( A^{4m+2} - A^{-2}\right) Id_{V}\]
\[ Id_{V} = \frac{1}{( A^{4m+4} - A^{-2})^N ( A^{4m+2} - A^{-2})^{2g-1}}\left( A^{4m+3} \Delta_{-,1} + A^{-1} \Delta_{+,m}\right)^N \circ \left( A^{4m+1} \Delta_- + A^{-1} \Delta_{+,m} \right)^{2g-1}\]
When expanding the last expression, we see that every summand has a factor of the form $\Delta_{-,1} \circ \Delta^{2g-1}_{-}$ or $\Delta^{N-1}_{+,m} $. Hence, by evaluating ${}_r D_a$, we obtain ${}_r D_a\in V^{\partial \Sigma}_m$ as desired. 
\end{proof}

\subsection{A generating set for $\mathcal{S}(\Sigma\times S^1)$}
To conclude the proof of finiteness for the Kauffman Bracket Skein Module of trivial $S^1$-bundles over surfaces with boundary, this section studies relations among non-separating simple closed curves.

\begin{lem}\label{lem_change_sides}
Any arrowed non-separating simple closed curve in $\Sigma$ can be written as follows in $\mathcal{S}(\Sigma\times S^1)$ 
\[\left( A-A^{-1}\right) \includegraphics[valign=c,scale=.4]{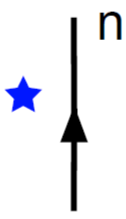} = A \includegraphics[valign=c,scale=.4]{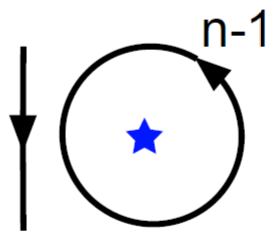} - A^{-1} \includegraphics[valign=c,scale=.4]{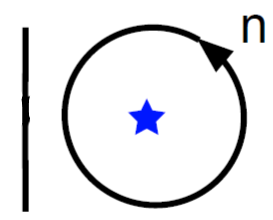}.\]
\end{lem}

\begin{proof}
Using the R5 relation, we obtain 
$\includegraphics[valign=c,scale=.35]{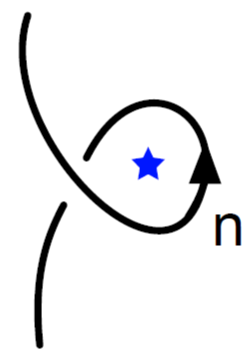}  = \includegraphics[valign=c,scale=.35]{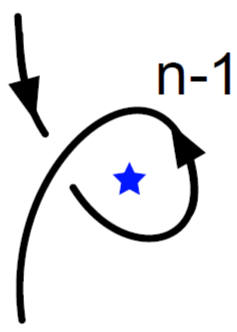}$. Thus, 
\[A \includegraphics[valign=c,scale=.4]{fig_37_n.png} - A^{-1} \includegraphics[valign=c,scale=.4]{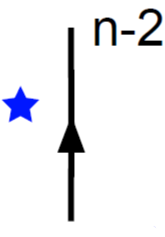} = A \includegraphics[valign=c,scale=.4]{fig_37_1n-1.png} - A^{-1} \includegraphics[valign=c,scale=.4]{fig_37_0n.png}.\]

Proposition 4.1 of \cite{KBSM_S1} states that non-separating curves with $n$ and $n-2$ arrows are the same in $\mathcal{S}(\Sigma\times S^1)$. Thus, the result follows.
\end{proof}

\begin{remark}\label{remark_change_sides}[Application of Lemma \ref{lem_change_sides}] 
Let $\gamma$ be a non-separating simple closed curve in $\Sigma$ and let $c\in \partial \Sigma$. Let $\widetilde \gamma$ be an arrowed diagram with one copy of $\gamma$ and some copies of $c$; think of $\gamma$ to be `on the right side' of $c$. Lemma \ref{lem_change_sides} states that, at the expense of adding more copies of $c$ and arrows, $\gamma$ is a linear combination of two diagrams where $\gamma$ is on the other side of $c$. 
\end{remark}

\begin{prop}\label{thm_boundary_case}
Let $\Sigma$ be an orientable surface of genus $g>0$ and $N>0$ boundary components. Let $D\subset \Sigma$ be a $(N+1)$-holed sphere containing $\partial \Sigma$, and let $\mathcal{F}$ be a collection of $2^{2g}-1$ non-separating simple closed curves in $\Sigma-D$ such that each curve in $\mathcal{F}$ represents a unique non-zero element of $H_1(\Sigma-D; \Z/2\Z)$. 
Let $\mathcal{B}$ be the collection $\{ \gamma \cup \alpha, U \cup \alpha \}$, where $\gamma$ is a curve in $\mathcal{F}$ zero or one arrow, $U$ is an arrowed unknot, and $\alpha$ is any collection of boundary parallel arrowed circles. 
Then $\mathcal{B}$ is a generating set for $\mathcal{S}(\Sigma\times S^1)$ over $\Q(A)$.
\end{prop}

\begin{proof}
By Proposition \ref{prop_genus_case}, we only need to focus on the non-separating curves.
Let $\widetilde \gamma$ be a non-separating simple closed curve in $\Sigma$. After using Lemma \ref{lem_change_sides} repeatedly, we can write $\widetilde \gamma$ as a linear combination of arrowed diagrams of the form $\gamma \cup \alpha$ where $\gamma$ is a non-separating curve in $\Sigma-D$ and $\alpha$ is a collection of boundary parallel curves. 
Observe that the work on Section 5 of \cite{KBSM_S1} holds for surfaces with connected boundary since generators for $\pi_1(S_g, *)$ and $Mod(S_g)$ also work for $S_{g,1}$. Thus, by Proposition 5.5 of \cite{KBSM_S1}, two non-separating curves $\gamma_1,\gamma_2 \subset \Sigma-D$ with the same number of arrows are equal in $\mathcal{S}(\Sigma\times S^1)$ if $[\gamma_1]=[\gamma_2]$ in $H_1(\Sigma-D; \Z/2\Z)$. The conditions on the number of arrows for non-separating curves follows from Propositions 4.1 of \cite{KBSM_S1}.
\end{proof}

\begin{thm}\label{thm_conj_3_boundary}
Let $\Sigma$ be an orientable surface with non-empty boundary. Then $\mathcal{S}(\Sigma\times S^1)$ is a finitely generated $\mathcal{S}(\partial \Sigma\times S^1, \Q(A))$-module of rank at most $2^{2g+1}-1$. 
\end{thm}
\begin{proof}
As a module over $\S(\partial M, \Q(A))$, we can overlook $\partial$-parallel subdiagrams. Proposition \ref{thm_boundary_case} implies that $\S(M)$ is generated by the empty diagram and diagrams in $\mathcal{F}$ with at most one arrow. 
\end{proof}


\section{Seifert Fibered Spaces}\label{section_SFS}

Seifert manifolds with orientable base orbifold can be built as Dehn fillings of $\Sigma\times S^1$ where $\Sigma$ is a compact orientable surface. A result of Przytycki \cite{fundamentals_KBSM} implies that their Kauffman bracket skein modules are isomorphic to the quotient of $\mathcal{S}(\Sigma\times S^1)$ by a submodule generated by curves in $\partial \Sigma \times S^1$ bounding disks after the fillings. In this section, we use these new relations to show the finiteness conjectures for a large family of Seifert fibered spaces. For details on the notation see next subsection.

\begin{thm}\label{thm_finite_SFS}
Let $M=M\left(g; b,\{(q_i,p_i)\}_{i=1}^n\right)$ be an orientable Seifert fibered space with non-empty boundary. Suppose $M$ has orientable orbifold base. Then, $\mathcal{S}(M)$ is a finitely generated $\mathcal{S}(\partial M, \Q(A))$-module of rank at most $(2^{2g+1}-1) \prod_{i=1}^n (2q_i-1)$.
\end{thm}

\begin{thm}\label{conj2_SFS}
Seifert fibered spaces of the form $M\left(g; 1,\{(1,p_i)\}_{i=1}^{n}\right)$ satisfy Conjecture \ref{conjecture_strong_finiteness}. In particular, Conjecture \ref{conjecture_strong_finiteness} holds for $\Sigma_{g,1}\times S^1$. 
\end{thm}


\subsection{Links in Seifert manifolds} 
Let $\Sigma$ be a compact orientable surface of genus $g\geq 0$ with $N\geq 0$ boundary components. Fix non-negative integers $n$, $b$ with $N=n+b$. Denote the boundary components of $\Sigma$ by $\partial_1, \dots, \partial_N$ and the isotopy class of a circle fiber in $\Sigma\times S^1$ by $\lambda = \{pt\}\times S^1$. For each $i=1, \dots, n$, let $(q_i,p_i)$ be pairs of relatively prime integers satisfying $0<q_i<|p_i|$. 
Let $M\left(g; b,\{(q_i,p_i)\}_{i=1}^n\right)$ be the result of gluing $n$ solid tori to $\Sigma\times S^1$ in such way that the curve $p_i [\lambda]+ q_i[\partial_i] \in H_1(\partial_i \times S^1)$ bounds a disk. In summary, $\Sigma$ is the base orbifold of the Seifert manifold, $n$ counts the number of exceptional fibers, and is $b$ the number of boundary components of the 3-manifold. 

Let $M$ be an orientable Seifert manifold with orientable orbifold base. It is a well known fact that $M$ is homeomorphic to some $M\left(g; b,\{(q_i,p_i)\}_{i=1}^n\right)$ \cite{hatcher_3M}. Links in $M$ can be isotoped to lie inside $\Sigma\times S^1$. Thus, we can represent links in $M$ as arrowed diagrams in $\Sigma$ with some extra Reidemester moves. 
By Proposition 2.2 of \cite{fundamentals_KBSM} and Proposition \ref{thm_boundary_case}, $\mathcal{S}(M)$ is generated by the family of simple diagrams $\mathcal{B}= \{ \gamma \cup \alpha, U \cup \alpha \}$. 

\begin{defn}
Let $D\in \mathcal{B}$. Let $l_i\geq 0$ be the number of parallel copies of $\partial_i$ in $D$. Let $\eps_i\geq 0$ be the number of arrows (regardless of orientation) among all components of $D$ parallel to $\partial_i$. If $D$ contains an unknot $U$, denote by $u\geq 0$ the number of arrows in $U$. If $D$ contains a non-separating loop, let $u=0$. 
The \textbf{absolute arrow sum} of $D$ is the total number of arrows among its separating loops $s:= u + \sum_i \eps_i$. 
$D$ is \textbf{standard} if $0\leq \eps_i \leq 1$ for every $i=1,\dots, N$; and such arrows (if exist) lie in the loop furthest from the boundary. 
\end{defn}

\begin{lem}\label{lem_std_diagrams_B}
Every diagram $D$ in $\mathcal{B}$ is a $\Z[A^{\pm1}]$-linear combination of standard diagrams $D'$ satisfying $s'\leq s$ and $l_i'\leq l_i, \forall i$.
\end{lem}
\begin{proof}
Follows from Proposition 4.1 of \cite{KBSM_S1} and Lemmas \ref{lem_popping_arrows}, \ref{lem_merging_unknots}, and \ref{lem_pushing_arrows}.
\end{proof}
The rest of this section is devoted to understand how the quantities $s$ and $l_i$ behave under certain relations in $\mathcal{B}$. We use Lemma \ref{lem_std_diagrams_B} implicitly to rewrite any relation in terms of standard diagrams with bounded sums $s$ and $l$. 

\begin{remark}[Moving arrows]\label{remark_pushing_arrows}
We think of Lemma \ref{lem_pushing_arrows} as a set of moves that change the arrows between consecutive circles at the expense of adding `debris' terms. Observe that $|b-a+2| \leq |a| + |b|$ as long as $b<0$ or $a>0$. In particular, the debris terms in the equations of Lemma \ref{lem_pushing_arrows} parts (i) and (iii) will have absolute arrow sums bounded above by the LHS whenever $b<0$ or $a>0$. The same happens with parts (ii) and (iv) when $b>0$ or $a<0$. This can be summarized as follows: ``We can move arrows between consecutive nested loops without increasing the arrow sum nor $l_i$."
\end{remark}

\subsubsection{Local moves around an exceptional fiber}\label{subsubsection_Omega_move}
Fix an index $i=1,\dots, n$. By construction, there is a loop $\beta_i$ in the torus $\partial_i \times S^1$ bounding a disk $B_i$ in $M$; $\beta_i$ homologous to $\left( p_i[\lambda] + q_i [\partial_i]\right) \in H_1(\partial_i \times S^1,\Z)$. Following \cite{KBSM_prism}, we can slide arcs in $\Sigma\times S^1$ over the disk $B_i$ and get new Reidemeister moves for arrowed projections in $\Sigma\times S^1$. We obtain a new move, denoted by $\Omega(q_i, p_i)$ (see Figure \ref{fig_Omega_move}).
\begin{figure}[h]
\centering
\includegraphics[scale=.35]{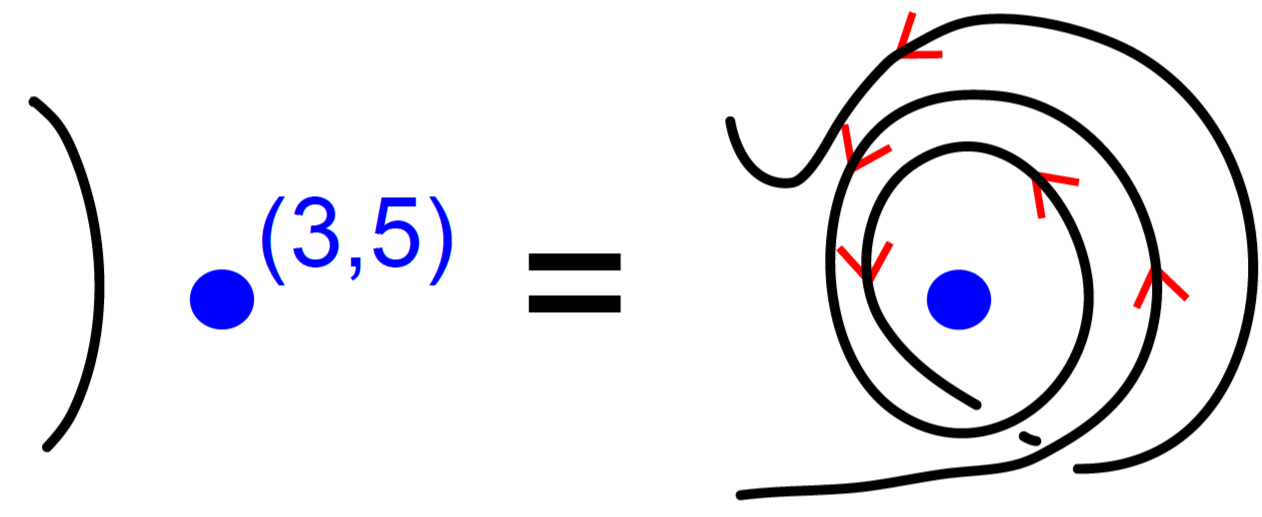}
\caption{$\Omega(q_i,p_i)$ is obtained by drawing $q_i$ concentric circles and $p_i$ arrows equidistributed. Notice that the orientation of the arrows in the RHS is determined by the sign of $p_i$.}
\label{fig_Omega_move}
\end{figure}

We can perform the $\Omega(q_i,p_i)$-move on an unknot near the boundary $\partial_i$ and resolve the $q_i-1$ crossings with K1 relations. 
Since $0<q_i<|p_i|$, there is only one state with orientations of the arrows not cancelling. This unique state has exactly $q_i$ concentric loops while the other states have strictly fewer loops and no more than $|p_i|-2$ arrows. 
We then obtain an equation in $\mathcal{S}(M)$ called the \textbf{$\Omega(q_i,p_i)$-relation}. Figure \ref{fig_Omega_relation} shows a concrete example of this equation.

\begin{remark}[The $\Omega(q_i,p_i)$-relation]\label{remark_Omega_move}
The $\Omega(q_i,p_i)$-relation lets us write a diagram with $q_i$ concentric loops and $|p_i|$ arrows arranged in a particular way as a $\Z[A^{\pm}]$-linear combination of diagrams with $0\leq l_i< q_i$ concentric circles and $0\leq \eps_i<|p_i|$ arrows (see Figure \ref{fig_Omega_relation}). The LHS has $|p_i|$ arrows oriented in the same direction depending on the sign of $p_i$; counterclockwise if $p_i>0$ and clockwise otherwise. Notice that the condition $0<q_i<|p_i|$ implies that every parallel loop in the LHS has at least one arrow. 
\end{remark}

The special arrangement of arrows in the LHS of the $\Omega(q_i,p_i)$-relation is important and depends on the pair $(q_i, p_i)$. In practice, we rearrange the arrows around the outer $q_i$ copies of $\partial_i$ to match with the LHS of the $\Omega(q_i,p_i)$-relation. Lemma \ref{lem_Omega_move} uses this idea in a particular setup. 

\begin{figure}[h]
\centering
$\includegraphics[valign=c,scale=.45]{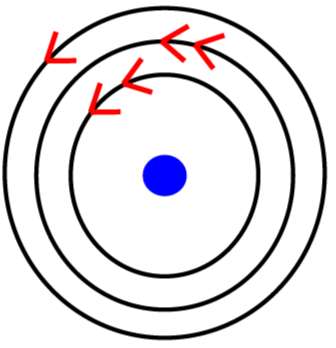} = A^2 \includegraphics[valign=c,scale=.45]{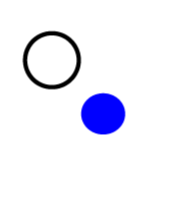} - A^{2} \includegraphics[valign=c,scale=.45]{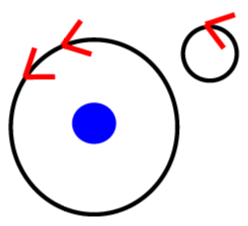} - A^2 \includegraphics[valign=c,scale=.45]{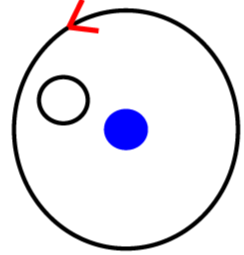} - A^4 \includegraphics[valign=c,scale=.45]{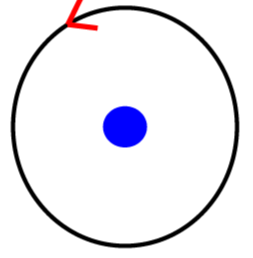}$
\caption{$\Omega(3,5)$-relation.}
\label{fig_Omega_relation}
\end{figure}

\begin{lem}\label{lem_Omega_move}
The following equation in $\mathcal{S}(M)$ relates identical diagrams outside a neighborhood of $\partial_1$. Let $D\in \mathcal{B}$ and $x\geq |p_1|$. 
Suppose that $l_1\geq q_1$, the loop furthest from $\partial_1$ has $x$ arrows with the same orientation as in the LHS of the $\Omega(q_1,p_1)$-relation, and no other loop in $D$ parallel to $\partial_1$ has arrows. Then $D$ is a sum of diagrams $D'\in \mathcal{B}$ with $l'_1<l_1$ and at most $x$ arrows.
\end{lem}

\begin{proof}
Rearrange the arrows to prepare for the $\Omega(q_1, p_1)$-relation using Lemma \ref{lem_pushing_arrows}. Remark \ref{remark_pushing_arrows} explains that the debris terms in this procedure will have arrow sum at most $x$ and $l'_1<l_1$. After performing the $\Omega(q_1,p_1)$-move, we obtain diagrams with lesser loops $l'_1 < l_1$. Observe that the lower arrow sum is explained due to at least one pair of arrows getting cancelled; this always happens since $0<q_1<|p_1|$. In particular, we lose at least two arrows when performing the move. 
\end{proof}

\subsubsection{Global relations}

We now discuss relations among elements in $\mathcal{B}$ of the form $U \cup \alpha$. 
Lemma \ref{lem_earthquake} permits us to add new loops around each $\partial_i$ all of which have one arrow of the same direction. This move is valid as long as we have enough arrows on the unknot $U$; i.e. $u\geq 4g+2N$. The debris terms are $\Z[A^{\pm1}]$-linear combinations of standard diagrams with fewer arrow sum and $l'_i\leq l_i +1$. This move is key to prove Theorem \ref{conj2_SFS}. 

Consider the decomposition $\mathcal{P}_+$ of $\Sigma$ described in Figure \ref{fig_global_P+}. Set $\partial_0$ to be the left-most unknot in $\mathcal{P}_+$ oriented counterclockwise. 
As we did in Definition \ref{defn_linear_pants}, if $v_i\in \Z$ we will draw one copy of $\partial_i$ with $v_i$ arrows oriented as in $\mathcal{P}_+$, and do nothing if $v_i = \emptyset$. For $v \in \left(\Z\cup \{\emptyset\}\right)^{N+1}$, denote by $E_{v}$ the diagram obtained by drawing $\partial_i$ with $v_i$ arrows on it. For example, $E_{(b, \emptyset, \dots, \emptyset)}$ corresponds to the arrowed unknot $S_b$. 
\begin{figure}[h]
\centering
\includegraphics[scale=.5]{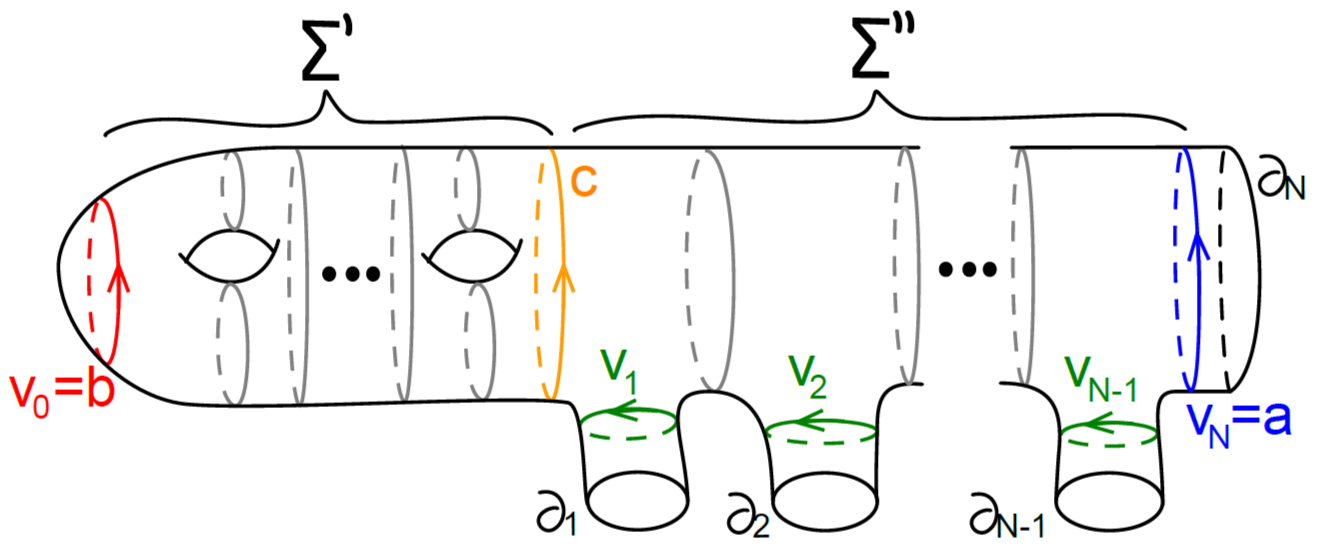}
\caption{$\mathcal{P}_+$ induces linear pants decompositions on $\Sigma''$ and sausage decompositions on $\Sigma'$.}
\label{fig_global_P+}
\end{figure}

We define the $\Delta$-maps from Definition \ref{defn_delta} on the family of diagrams $E_v$ with exactly one of $v_0$ and $v_N$ being empty. If $v_0=\emptyset$ and $v_N\in \Z$, define $s(E_v)=E_{(v_0, \dots, v_{N-1}, v_{N}+2)}$. If $v_0\in \Z$ and $v_N=\emptyset$, define $s(E_v)=E_{(v_0+2, \dots, v_{N-1}, v_{N})}$. 

\begin{lem} \label{lem_earthquake}
Let $a\geq 0$. The following equation in $\mathcal{S}(\Sigma\times S^1)$ holds modulo $\Z[A^{\pm1}]$-linear combinations of standard diagrams 
$E_{(\emptyset, \dots, \emptyset, a')}$ and $E_{(a'', 1,1, \dots, 1, \emptyset)}$ with $a'$, $a''+(N-1)$ integers in $[0,a+4g+2N-2)$.
\[E_{(\emptyset, \dots, \emptyset, a+4g+2N-2)} \cong (-1)^{N-1} A^{-4g-2N+4} E_{(a+4g+N-1, 1,1, \dots, 1, \emptyset)}.\]
\end{lem} 

\begin{proof}
Observe first that $\mathcal{P}_+$ induces a linear pants decomposition on $\Sigma''$ as in Section \ref{section_planar_case}. Here, a copy of $\partial_N$ with $x\in \Z$ arrows, $E_{(\emptyset, \dots, \emptyset, x)}$, corresponds to the diagram $D^{N-1}_{x,(\emptyset, \dots, \emptyset)}$. Equation \eqref{eqn_1_planar} of Lemma \ref{equation_1_planar} with $n=k_0=N-1$ states the following 
\[  \Delta^{N-1}_+ \left( D^{N-1}_{a, (\emptyset, \dots, \emptyset)} \right)= \sum_{e \in \{0,1\}^{N-1}} (-1)^{o(e)} A^{z(e) - o(e)} D^{0}_{a+o(e), e} \quad . \] 
For any $x\in \Z$ and $v\in \{0,1\}^{N-1}$, the diagram $D^0_{x,v}$ contains a copy of the curve $c$ (see Figure \ref{fig_global_P+}) with $x$ arrows. 
Now, observe that $\mathcal{P}_+$ also induces a sausage decomposition of $\Sigma'$ (see \cite{KBSM_S1}). Using the notation in Section 3.3 of \cite{KBSM_S1}, the part of the diagram $D^0_{x,v}$ inside the subsurface $\Sigma'\subset \Sigma$ is denoted by $D^{2g}_{x}$. Proposition 3.13 of \cite{KBSM_S1} implies the equation $\Delta^{2g}_+(D^{2g}_x) = \Delta^{2g}_{-}(D^{0}_x)$, where $D^0_x$ is a copy of the left-most unknot $\partial_0$ (red loop) in $\mathcal{P}_+$ with $x$ arrows. Putting everything together, we obtain the following relation in $\mathcal{S}(\Sigma\times S^1)$:
\begin{equation*}
\Delta^{2g+N-1}_+ (E_{(\emptyset, \dots, \emptyset,a)}) = \sum_{e \in \{0,1\}^{N-1}} (-1)^{o(e)} A^{z(e) - o(e)} \Delta^{2g}_- (E_{(a+o(e), e_1, \dots, e_N, \emptyset)}).
\end{equation*}
The result follows by taking the summads on each side with the most number of arrows. 
\end{proof}

\subsection{Proofs of Theorems \ref{thm_finite_SFS} and \ref{conj2_SFS}}\label{section_proof_41}
Recall that $\mathcal{S}(M)$ is generated by all standard diagrams, and such diagrams are filtered by the complexity $(s,\sum_i l_i)$ in lexicographic order. Here, $s=u+\sum_i \eps_i$ is the absolute arrow sum and $\sum_i l_i$ is the number of boundary parallel loops. Throughout the argument we will have debris terms with lower complexity $(s',\sum_i l'_i)$; on each of those terms, we can perform a series of combinations of Lemmas \ref{lem_popping_arrows},  \ref{lem_flip_unknot}, \ref{lem_merging_unknots}, and \ref{lem_pushing_arrows} in order to write them in terms of standard diagrams with complexities $s''\leq s'$ and $l''_i\leq l'_i$.

Let $D\in \mathcal{B}$ be a diagram. Suppose that $D$ is of the form $D=\gamma \cup \alpha$, where $\gamma$ is an non-separating simple closed curve with at most one arrow and $\alpha$ is a collection of arrowed boundary parallel loops. We can rewrite $D$ in $\mathcal{S}(M)$ as $D=\frac{-1}{(A^2+A^{-2})} (D\cup U)$ where $U$ is a small unknot with no arrows. Proposition \ref{lem_case_2} focuses on the subdiagram $U\cup \alpha$ near a fixed boundary component. 

\begin{prop}\label{lem_case_2}
Let $D\in \mathcal{B}$ be a standard diagram with $l_{i_0}\geq q_{i_0}$ for some ${i_0}\in \{1,\dots, n\}$. Then $D$ is a linear combination of some standard diagrams $D'$ identical to $D$ outside a neighborhood of $\partial_{i_0}$, satisfying 
\[l'_{i_0}< l_{i_0} \text{ and } u'+\eps'_{i_0}\leq 2(u + |p_{i_0}|). \]
\end{prop} 
\begin{proof}
For simplicity, set ${i_0}=1$. We assume that $p_1>0$ so that the arrows in the LHS of the $\Omega(q_1,p_1)$-relation are oriented counterclockwise; the other case is analogous. 
We can assume that if $\eps_1=1$, then the orientation of the arrow in the loop furtherst from $\partial_1$ agrees with the LHS of the $\Omega(q_1,p_1)$-relation. This is true since Lemma \ref{lem_crossing_borders} lets us flip the orientation at the expense of having one debris diagram with $l'_1=l_1$, $\eps'_1=0$, and $u'=u+1$. 

Denote by $D_x$ the standard diagram in $\mathcal{B}$, identical to $D$ away from a neighborhood of $\partial_1$ with $l_1$ copies of $\partial_1$, having $x$ arrows oriented counterclockwise in the loop furtherst from $\partial_1$. 
Recall that $S_a$ denotes a small unknot with $a\in \Z$ arrows oriented counterclockwise. We have that $D=D_{\eps_1}\cup S_u$, where the disjoint union of the diagrams is made so that $S_u$ lies inside a small disk away from the diagram $D_{x}$. 

Merge the arrows on $U$ with the outer loop around $\partial_1$ using Lemma \ref{lem_merging_unknots}. Thus, $D$ is a linear combination of diagrams $D_x$ with no unknots ($U=\emptyset$). 
If $\eps_1=1$, we get diagrams with $0\leq x\leq u+\eps_1$, and if $\eps_1=0$, we obtain diagrams with $0\leq |x|\leq u$. 
We focus on each $D_x$.
Use the relation around $\partial_1$
\begin{equation}\label{add_arrows}
\includegraphics[valign=c,scale=.45]{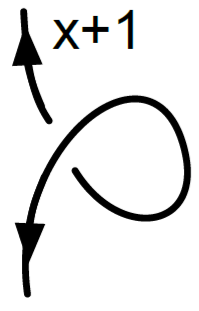} = \includegraphics[valign=c,scale=.45]{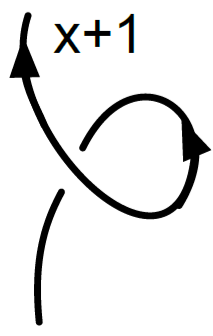}
\quad \implies \quad 
\includegraphics[valign=c,scale=.55]{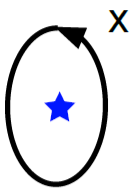} = -A^2 \includegraphics[valign=c,scale=.55]{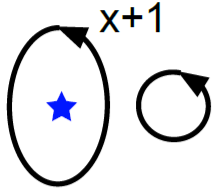} - A^{4} \includegraphics[valign=c,scale=.55]{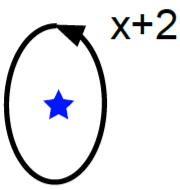}
\end{equation}
to write $D_x$ as a linear combination of $D_{x+1}\cup S_1$ and $D_{x+2}$. Thus, at the expense of getting a cluster of 1-arrowed unknots $S_{\pm1}$, we can increase/decrease the arrows in the outermost loop around $\partial_1$. Hence, the original diagram $D$ is a linear combination of diagrams of the form $D_{x} \cup \left( \cup_y S_1\right)$ where $x\geq p_1$, $y\geq 0$ and $x+y\leq 2(u+ p_1)$. To see the upper bound for $x+y$, observe that if we start with $D_{-u}$, one might need to add a copy of $S_{1}$ $(u+ p_1)$ times in order to reach $x \geq p_1$. 
Lemma \ref{lem_Omega_move} implies that each $D_{\pm x} \cup \left( \cup_y S_1\right)$ is a linear combination of diagrams with $l'_1<l_1$ and at most $x+y$ arrows. After making such diagrams standard and merging the arrowed unknots, we obtain diagrams with $l'_1<l_1$ and $u'+\eps'_1\leq 2(u+  p_1)$ as desired.  
\end{proof}

\begin{proof}[Proof of Theorem \ref{thm_finite_SFS}]
Let $M=M\left(g; b,\{(q_i,p_i)\}_{i =1}^n\right)$ be a Seifert fibered space with non-empty boundary. 
Proposition \ref{thm_boundary_case} and Lemma \ref{lem_std_diagrams_B} imply that $\mathcal{S}(M)$ is generated over $\Q(A)$ by standard diagrams in $\mathcal{B}$. 
Furthermore, it follows from Lemmas \ref{lem_merging_unknots}, \ref{lem_pushing_arrows}, and \ref{lem_crossing_borders} that it is enough to consider standard diagrams with all arrows on separating loops oriented counterclockwise. 
Notice that the standard condition allow us to overlook the numbers $l_{n+j}$ for $j=1,\dots, b$ since they correspond to coefficients of the ring $\S(\partial M, \Q(A))$. 

Divide the collection $\mathcal{B}$ into two sets $\mathcal{B}_{ns} =\{\gamma\cup \alpha\}$  and $\mathcal{B}_{U}=\{U\cup \alpha\}$. 
Proposition 4.1 of \cite{KBSM_S1} implies that arrowed non-separating simple closed curves are equal in $\mathcal{S}(M)$ if they are the same loop and have the same number of arrows modulo 2. Thus, using Proposition \ref{lem_case_2}, we obtain that $\Q(A)\cdot\mathcal{B}_{ns}$ is generated by standard diagrams $D=\gamma \cup \alpha$ with $0\leq l_i< q_i$ for $i=1,\dots, n$ and all arrows in copies of $\partial$-parallel loops oriented counterclockwise. 
Hence, $\Q(A)\cdot\mathcal{B}_{ns}$ is generated as a $\S(\partial M, \Q(A))$-module by a set of cardinality 
\[ r_{ns} \leq (2^{2g+1}-2) \prod_{i=1}^n (2q_i-1).\]
Proposition \ref{lem_case_2} implies that $\Q(A)\cdot\mathcal{B}_{U}$ is generated over $\Q(A)$ by standard diagrams satisfying $0\leq l_i <q_i$ for all $i=1, \dots, n$. Therefore, since $U$ can be pushed towards the boundary, $\Q(A)\cdot\mathcal{B}_{U}$ is generated over $\S(\partial M ,\Q(A))$ by a finite set of cardinality 
\[ r_U \leq  \prod_{i=1}^n (2q_i-1),\]
Hence, $\S(M)$ is a finitely generated $\S(\partial M, \Q(A))$-module.
\end{proof}


\begin{proof}[Proof of Theorem \ref{conj2_SFS}]
Let $\lambda\subset \partial M$ be a $S^1$-fiber and let $\mu_N =\partial_N \times \{pt\}$ be a meridian of $\partial M$. 
For $i=1, \dots, n$, the $\Omega(1,p_i)$-move turns loops parallel to $\partial_i$ into arrowed unknots. Thus, Proposition \ref{thm_boundary_case}, Lemma \ref{lem_std_diagrams_B}, and Equation \eqref{eqn_23} imply that $\S(M)$ is generated over $\Q(A)$ by standard diagrams in $\mathcal{B}=\{\gamma\cup \alpha, U\cup \alpha\}$ with no parallel loops around the exceptional fibers. In particular, $\alpha$ only contains loops around $\partial_N$. Hence $\Q(A) \cdot \mathcal{B}_{ns}$ is generated over $\Q(A)[\mu_N]$  by elements of the form $\gamma$ and $\gamma \cup \alpha$ where $\gamma \in \mathcal{F}$ has at most one arrow and $\alpha$ is a copy of $\partial_N$ with one arrow. 

Let $U\cup \alpha\in \mathcal{B}_U$ and suppose that $U$ has $u\neq 0$ arrows. Using Equation \eqref{add_arrows} of Proposition \ref{lem_case_2}, we can assume that the loop of $\alpha$ furthest to the boundary has at least one arrow. Then, using Lemmas \ref{lem_flip_unknot} and \ref{lem_merging_unknots}, we can write any diagram in $\mathcal{B}_U$ as a $\Q(A)$-linear combination of diagrams with only $\partial$-parallel curves and such that the loop furthest to $\partial_N$ has $x\geq 0$ arrows oriented clockwise. In other words, $\Q(A)\cdot \mathcal{B}_U=\Q(A) \langle U, \mu_N^k\cdot \alpha_x| k, x\geq 0\rangle$, where $\alpha_x$ denotes a copy of $\mu_N$ with $x$ arrows. 

We will see that it is enough to consider $0\leq x<4g+2n$. Take $\mu_N^k\cdot \alpha_x$ with $k\geq 0$ and $x\geq 4g+2n$. By Lemma \ref{lem_earthquake}, $\mu_N^k\cdot \alpha_x$ is a $\Z[A^{\pm1}]$-linear combination of diagrams of the form $U\cup \mu_N^k$ and $\mu_N^k \cdot \alpha_{y}$ with $0\leq y<x$. We can proceed as in the previous paragraph and write the diagrams $U\cup \mu_N^k$ as $\Z[A^{\pm1}]$-linear combinations of $\mu_N^{\max(0,k-1)}\cdot \alpha_{x'}$ for some $x'\geq 0$. Hence, $\Q(A)\cdot \mathcal{B}_U=\Q(A) \langle U, \mu_N^k\cdot \alpha_x| 0\leq k, 0\leq x <4g+2n \rangle$.

To end the proof, consider $F_1$ the subspace generated by $\mathcal{B}_{ns} \cup \{\mu_N^k\cdot \alpha_x| 0\leq x <4g+2n \}$ over $\Q(A)$, and $F_2$ the $\Q(A)$-subspace generated by arrowed unknots. By Proposition \ref{thm_boundary_case}, $\S(M) = F_1 + F_2$. Let $\Sigma_1$ and $\Sigma_2$ be neighborhoods of $\mu_N$ and $\lambda$ in $\partial M$, respectively. We have shown that $F_1$ is a $\S(\Sigma_1, \Q(A))$-module of rank at most $2(2^{2g+1}-2)+ 4g+2n$. Also, since every arrowed unknot can be pushed inside a neighborhood of $\Sigma_2$, $F_2$ is generated over $\S(\Sigma_2, \Q(A))$ by the empty link. So $F_2$ is a $\S(\Sigma_2, \Q(A))$-module of rank at most one. 
\end{proof}


\bibliography{KBSM_of_SFS}
\bibliographystyle{plain}

\begin{flushright}
Jos\'e Rom\'an Aranda, University of Iowa\\
email: \texttt{jose-arandacuevas@uiowa.edu}\\ $\quad$ \\
Nathaniel Ferguson, Colby College \\ 
email: \texttt{nrferg21@colby.edu}
\end{flushright}
\end{document}